\documentclass[12pt]{article}

\usepackage{amssymb}
\usepackage{amsmath}
\usepackage{graphicx}
\usepackage{epsfig}
\usepackage{subfigure}
\usepackage{color}

\def\shadowbox{\hbox{\rule[-0.0ex]{0.1ex}{1.2ex}%
\hspace{-0.1ex}\rule[-0.0ex]{1.2ex}{0.1ex}%
\hspace{0.0ex}\rule[-0.0ex]{0.1ex}{1.2ex}\hspace{-1.3ex}%
\rule[1.15ex]{1.25ex}{0.1ex}\hspace{-0.0ex}\rule[-0.25ex]{0.3ex}{1.1ex}%
\hspace{-1.2ex}\rule[-0.25ex]{1.1ex}{0.25ex}}}
\def\qed{\ifmmode \hbox{\hfill\shadowbox}
     \else \hphantom{x}\hfill\shadowbox \fi}

\newtheorem{theorem}{Theorem}[section]
\newtheorem{lemma}[theorem]{Lemma}
\newtheorem{definition}[theorem]{Definition}
\newtheorem{proposition}[theorem]{Proposition}

\newcommand{\nc}{\normalcolor}

\def\R {\mathbb{R}}
\def\C {\mathbb{C}}
\def\G {\mathcal{G}}
\def\Cp {\mathbb{C}^+}
\def\E {\mathbb{E}}
\def\Eek {\mathcal{E}^{(\mu)}_{\kappa,\eta}}
\def\Eck {\mathcal{E}^{(\mu)}_{\kappa}}
\def\EMe {\mathcal{E}^{(\lambda)}_{\kappa,\eta}}
\def\EMc {\mathcal{E}^{(\lambda)}_{\kappa}}
\def\P {\mathbb{P}}
\def\N {\mathcal{N}}

\def\To {\mathcal{T}_1}
\def\Tt {\mathcal{T}_2}
\def\Ieta {I_{\eta}}

\def\X {\mathcal{X}}
\def\B {\mathcal{B}}
\def\oa{\frac{1}{a}}
\def\on {\frac{1}{n}}

\def\Ys {Y^*}

\def\YYsh {(YY^*)^{-1/2}}
\def\dj {d_{j}}
\def\xj {x_{j}}

\def\bo {\overline{b}}
\def\rj {r_j}
\def\rjs {r_{j}^*}

\def\mn {m_n}
\def\mnu {\underline{m}_n}
\def\munder {\underline{m}}
\def\mun {\underline{m}}
\def\munp {\underline{m}_p}

\def\mj {m_{(j)}}
\def\mju {\underline{m}_{(j)}}
\def\Bj {B_{(j)}}

\def\Bn {B_n}
\def\tr {\textnormal{tr}}

\def\Ti {T^{-1}}
\def\Tih {T^{-\frac{1}{2}}}

\def\mM {m_{M}}
\def\mP {m_{p}}
\def\lm {\lambda_-}
\def\lp {\lambda_+}

\def\XXs {XX^*}
\def\Xs {X^{*}}
\def\YYs {YY^{*}}

\def\onYYsi {\left(\frac{1}{n}YY^{*}\right)^{-1}}
\def\Ys {Y^*}

\def\md {m_{\delta}}
\def\mud {\underline{m}_{\delta}}
\def\Th {T^{1/2}}
\def\Bjzi {(B_{(j)}-zI)^{-1}}

\def\Bnzi {(B_n-zI)^{-1}}

\def\oh {\frac{1}{2}}

\def\mlow {\mu_{-}}
\def\mhigh {\mu_{+}}
\def\fp {f_{p}}
\def\oox {\frac{1}{1+x}}

\def\finv {f_{\textnormal{Inv}}}
\def\mk {\mu_k}

\def\Matrix {\left( \Xn\Xn^*+\Yn\Yn^*\right)^{-\oh}\Yn\Yn^*\left( \Xn\Xn^*+\Yn\Yn^*\right)^{-\oh} }
\def\Matrix {\left( XX^*+YY^*\right)^{-\oh}YY^*\left( XX^*+YY^*\right)^{-\oh} }

\def\Xo {X_{(1)}}
\def\Xj {X_{(j)}}
\def\sm {s_{\min}}
\def\sM {s_{\max}}
\def\ot {\overline{t}}
\def\otau {\overline{\tau}}

\def\fdl {\finv (\lambda)d\lambda}

\def\hdl {\frac{\lambda}{1+\lambda\mud(z)}}
\def\hpl {\frac{\lambda}{1+\lambda\munp(z)}}
\def\fdl {\finv (\lambda)d\lambda}

\def\am {a_-}
\def\aM {a_+}

\def\Kz {K_0}

\allowdisplaybreaks[4]
\numberwithin{theorem}{section}
\begin{document}

\setcounter{tocdepth}{1}

\title{\Large Local Eigenvalue Density for General MANOVA Matrices}

\author{L\'aszl\'o  Erd\H{o}s\thanks{Institute of Science and Technology Austria, Am Campus 1, Klosterneuburg, 
A-3400, Austria. On leave from Institute of Mathematics, University of
Munich, Theresienstr. 39, D-80333 Munich, Germany.
Email: lerdos@ist.ac.at. Partially supported by SFB-TR12 Grant of the
German Research Council.}
$\;$and  Brendan Farrell\thanks{Computing and Mathematical Sciences, MC 305-16, 
California Institute of Technology,
1200 E. California Blvd., Pasadena, CA 91125, USA. \newline Email: farrell@cms.caltech.edu. Partially supported by Joel A. Tropp under ONR awards N00014-08-1-0883 and
N00014-11-1002 and a Sloan Research Fellowship. } }

\maketitle

\begin{abstract}
We consider random $n\times n$ matrices of the form $$\Matrix,$$ where $X$ and $Y$ have independent entries
with zero mean and variance one.
These matrices are the natural generalization of the Gaussian case, which are known as
 MANOVA matrices and which have joint eigenvalue density 
given by the  third  classical ensemble, the Jacobi ensemble. 
We show that, away from the spectral edge, the eigenvalue density converges
to  the limiting density of the Jacobi ensemble even 
on the shortest possible  scales of order $1/n$ 
 (up to $\log n$ factors).  This  result is the 
 analogue of the  local  Wigner semicircle law
and the  local  Marchenko-Pastur law for general MANOVA matrices.
\end{abstract}

{\bf AMS Subject Classification:} 15B52,  62H86

\vspace{.2in}

{\bf Keywords:}  MANOVA random matrix, Jacobi ensemble, Local density 
of eigenvalues.

\section{Introduction}

The three classical families of eigenvalue distributions of  Gaussian random matrices
are the Hermite, Laguerre and Jacobi ensembles. Hermite ensembles correspond
to Wigner matrices, $X=X^*$; Laguerre ensembles describe  sample covariance matrices, $XX^*$.
The random $n\times n$ matrices yielding the Jacobi ensembles have the form
\begin{equation}
\Matrix, \label{defmatrix}
\end{equation}
where  $X$ and $Y$ are $n\times [b n]$ and $n\times [a n]$ matrices with independent standard Gaussian entries.
Here  $a,b>1$ are fixed parameters of the model, $n$ is a large number, eventually
tending to infinity, and $[\cdot]$ denotes the integer part. The matrix entries can be real, complex
or self-dual quaternions, corresponding to the three symmetry classes, commonly distinguished
by the parameter $\beta=1,2,4$, respectively. The results in this paper are insensitive to
the symmetry class and for simplicity
we will consider the complex case $(\beta=2)$.

Matrices of the form~\eqref{defmatrix} are used in statistics for multivariate analysis
 of variance to determine correlation coefficients  (Section 3.3 of~\cite{Mui82}). 
This analysis is called MANOVA, though it has been largely limited to the special 
case when the entries of~\eqref{defmatrix} are Gaussian.

In this paper we address the case when the entries of $X$ and $Y$ in~\eqref{defmatrix} are independent
but have general distribution with zero mean and unit variance.
 In particular, the matrix entries are not required to be identically distributed.
We will call such a matrix with general entries a \emph{general MANOVA matrix}. 

Similarly to the Wigner and sample covariance matrices, the  joint eigenvalue density of~\eqref{defmatrix} is
explicitly known only for the Gaussian case. When the entries are standard complex Gaussians,
 it is given by
\begin{equation}
\textnormal{density}(\lambda_1,\ldots,\lambda_n)
=C_{a,b,n}\prod_{j=1}^n\lambda_j^{(a-1)n}(1-\lambda_j)^{(b-1)n}\prod_{1\leq j<k\leq n}|\lambda_j-\lambda_k|^2,\;\;\lambda_j\in[0,1],\label{jacobi}
\end{equation}
where $C_{a,b,n}$ is a normalizing constant.  The density has a similar form 
with different exponents when the matrix entries are real, or self-dual quaternions, see Section 3.6 of~\cite{For10}. 
Equation~\eqref{jacobi} defines the Jacobi ensemble, 
where the name refers to the form of the polynomial term in front of the Vandermonde determinant in \eqref{jacobi}.

The empirical density of the eigenvalues of~\eqref{defmatrix} -- equivalently,
the one-point correlation function of \eqref{jacobi} -- 
 converges almost surely,  as $n\rightarrow \infty$ , to 
the distribution with density given by 
\begin{equation}
 f_M(x)=  (a+b)\frac{\sqrt{(x-\lm)(\lp-x)}}{2\pi x(1-x)}\cdot I_{[\lm,\lp]}(x)\label{densityM}
\end{equation}
where 
\begin{equation}
 \lambda_{\pm}=\left(\sqrt{\frac{a}{a+b}\left(1-\frac{1}{a+b}\right)}\pm\sqrt{\frac{1}{a+b}
\left(1-\frac{a}{a+b}\right)}\right)^2.\label{lpm}
\end{equation}
The density $f_M$ was determined by Wachter~\cite{Wac80} and is discussed in Section 3.6 of~\cite{For10}. 
Note that $\lambda_{\pm}\in  (0,1) $, so that $f_M$ is supported on a 
compact subinterval of $(0,1)$. 
We will refer to $f_M(x)$ as the \emph{limiting distribution} of the eigenvalues of~\eqref{defmatrix} 
or as the MANOVA distribution. 

While the joint eigenvalue density \eqref{jacobi} is valid only for the Gaussian case, the
limiting empirical density is expected to be correct for general distributions  as well,
similarly to the universality of the Wigner semicircle law for Wigner matrices
or the Marchenko-Pastur (MP) law for sample covariance matrices.
Thus, general MANOVA matrices, the Jacobi ensemble and the distribution $f_M$ constitute a 
triplet analogous to general Wigner matrices, 
the Hermite ensemble and the semicircle law or 
sample covariance matrices, the Laguerre ensemble and the Marchenko-Pastur law.

Universality results have been intensely pursued for the latter two types 
of matrices, starting from the fundamental work of Wigner~\cite{Wig55} and Marchenko-Pastur~\cite{MP67}
who identified the corresponding distributions. These first results were on the macroscopic scale; the
empirical density on spectral scales containing $O(n)$ eigenvalues were shown to converge
in a weak sense to the limiting law. Recently local versions of these
fundamental laws have also been established on the shortest possible scale, containing
$O(\varphi(n))$ eigenvalues, where $\varphi(n)$ is a factor logarithmic in $n$.
For Wigner matrices it was achieved first in the bulk~\cite{ESY09, ESY10}
then optimally up to the edges~\cite{EYY11a}. For sample covariance matrices
the optimal scale in the bulk was reached in~\cite{ESYY10}, followed
by the optimal result up to the edge in~\cite{PY13}. Related results were
also obtained in~\cite{TV11b,TV11a,Pec12}. 

In this paper we prove  the local convergence of the density 
on the optimal scale  for the general MANOVA ensembles
in the bulk spectrum. This establishes the analogue
of the results \cite{ESY10, ESYY10} for these ensembles.
We remark that the convergence even on the largest
scale, i.e. the analogue of~\cite{Wig55, MP67}, has not
been known before  although it would essentially follow
from~\cite{SB95} if combined with the recent result in~\cite{PY13}.
   The main novelty of the current paper is
 the effective  stability analysis of the self-consistent equation 
for the Stieltjes transform of the density \eqref{implicitequation}.

Precise results on the local density have opened up the route
to establish the full universality of local eigenvalue statistics
for Wigner and sample covariance matrices, including 
precise identification of the statistics of consecutive gaps.
A new general method based   on the Dyson Brownian motion (DBM)
was first introduced in~\cite{ESY11}.  It is applicable to all symmetry classes~\cite{ESYY10},
to  very general distributions~\cite{EYY11a} and to sample covariance matrices~\cite{PY13}.
The local semicircle law (or the local MP law) is a basic input in all these works.
Local density results have also inspired an alternative
route to universality~\cite{TV11a, TV11b} that is applicable for the
complex case, $\beta=2$. 

In light of these developments for the Wigner and sample covariance matrices, 
the current work is the first step towards establishing 
the full universality of eigenvalue statistics
for the general MANOVA ensemble.

\section{Statement of the Main Result}

Given two positive constants  $\gamma=(\gamma_1, \gamma_2)$, we say that a
complex random variable $Z$ is $\gamma$-subexponential if it satisfies the following conditions:
\begin{equation}
\left\{ \begin{array}{l}
\E\; Z=0\\
\E \;|Z|^2=1\\
\P(|Z|\geq t^{\gamma_1})\leq \gamma_2 e^{-t}\;\;\textnormal{for all}\;\; t>0.\label{subexp}
\end{array}\right.
\end{equation}
A set of random variables is uniformly $\gamma$-subexponential 
if each random variable is $\gamma$-subexponential for a common $\gamma$.  Assuming that the matrix
elements of  $X$ and $Y$ are independent,  uniformly $\gamma$-subexponential random variables,
we will prove that the empirical distribution of the eigenvalues of \eqref{defmatrix} in the bulk 
converges on small scales to~\eqref{densityM} as $n\to \infty$.

The main tool for this approach is the  Stieltjes transform. 
The Stieltjes transform of a real random variable with distribution function $F$ 
is a function $\C^+\rightarrow\C^+$ defined by
\begin{equation}\label{ST}
m(z)=\int \frac{1}{t-z}dF(t).
\end{equation}
If the random variable has a density, then we also refer to the Stieltjes transform of the density. 
The  Stieltjes transform of $f_M$ is 
\begin{equation}\label{mN}
\mM(z)=  \frac{(2-a-b)z+a-1+\sqrt{ (a+b)^2z^2-(a+b)(2(a+1)-\frac{a}{a+b})z+(a-1)^2}  }{2 z(1-z)}.
\end{equation}
This formula is derived in Appendix~\ref{mNcomp}. \nc 

For self-adjoint matrices, we misuse notation and refer to the function 
\begin{equation*}
m_A(z)=\on \tr (A-zI)^{-1}
\end{equation*}
as the Stieltjes transform of the self-adjoint, $n\times n$ matrix $A$. 
If $\lambda_1,\ldots,\lambda_n$ are the eigenvalues of $A$, then we equivalently have 
\begin{equation*}
m_A(z)=\on\sum_{k=1}^n\frac{1}{\lambda_k-z},
\end{equation*}
which is the Stieltjes transform of the empirical measure.

Our main result shows that the eigenvalues of the general MANOVA matrix behave close to what is indicated by $\mM(z)$ and 
$f_M$ in the bulk with high probability.  
To state the result, we must formalize the term \emph{bulk}.  
Following this definition we state the  main theorem, Theorem~\ref{thmmanova}. 

\begin{definition}
Let  $\lp$ and $\lm$ be as given in~\eqref{lpm}.  
Define
\begin{equation*}
\EMe:=\left\{E+i\eta \in\C^+:\;E\in (\lm,\lp)\;\textnormal{ and }\;(\lp-E)(E-\lm)\geq \kappa\right\}
\end{equation*}
and set $\EMc=\mathcal{E}^{(\lambda)}_{\kappa,0}$. 
\end{definition}

\begin{theorem}\label{thmmanova}  Fix two real parameters $a, b>1$.
Let $X$ be an $n\times an$ random matrix and let $Y$ be an $n\times b n$ random matrix independent of $X$.
We assume that both matrices have  independent entries satisfying~\eqref{subexp} for a common 
$\gamma = (\gamma_1, \gamma_2)$. 
Let $m_{n,M}(z)$ be the Stieltjes transform of  the general MANOVA matrix 
\begin{equation}
\Matrix\label{matthm}. 
\end{equation}

\noindent $i)$ 
Then   for  any $\kappa, \eta >0$ with
 $\eta>\frac{1}{n\kappa^2}(\log n)^{2C\log\log n}$, we have
\begin{equation}\label{mmain}
 \P\left(\sup_{z\in\EMe} |m_{n,M}(z)-m_M(z)|> \frac{(\log n)^{C\log\log n}}{\sqrt{\eta n  \kappa}}\right)<n^{-c\log\log n}
\end{equation}
for all $n\ge n_0$ large enough and for  constants $C,c>0$.  Here
$n_0$, $C$ and $c$ depend only  on $\gamma$. 

\noindent $ii)$ Let $\mathcal{N}_{\eta}(E)$ denote the number of eigenvalues of~\eqref{matthm} 
contained in $[E-\frac{\eta}{2},E+\frac{\eta}{2}]$ 
and assume  $\eta \ge \frac{1}{n\kappa^2}(\log n)^{3C\log\log n}$. 
Then
\begin{equation}\label{Nmain}
\P\left( \sup_{E\in\EMc }\left|\frac{\mathcal{N}_{\eta}(E)}{n\eta}-f_M(E)  \right| > 
 \frac{(\log n)^{C\log\log n}}{  (\eta n \kappa)^{1/4} }  \right)\leq n^{-c\log\log n}.
\end{equation}
 \end{theorem}
We note that the entries of the matrices $X$ and $Y$ are not necessarily identically distributed.

\medskip

Theorem~\ref{thmmanova} shows that the  Stieltjes transform of the general MANOVA matrices 
is close to $m_M(z)$ 
when  the real part of $z$, $E=\Re z$,  is away from the spectral edge and the imaginary part $\eta=\Im z$ is small. 
In particular, $\eta$ may be as small as the 
shortest possible scale $1/n$, up to logarithmic corrections.
The second part of Theorem~\ref{thmmanova} is an easy consequence
of the first one and it asserts that the local density
on scale $1/n$ (modulo logarithmic corrections) is
given by the Manova density $f_M(E)$. 
While our analysis is valid down to the optimal scale  $\eta\gtrsim 1/n$,  the error bound
of the form $(n\eta\kappa)^{-1/2}$ is not optimal. The best estimate
should scale like  $(n\eta)^{-1}$ and should not blow up near the edge,
similarly to the best estimate in the Wigner case~\cite{EYY11a} and in the sample
covariance case~\cite{PY13}. 
Work to obtain the optimal error bounds is in preparation. 

\subsection{General Approach}\label{generalapproach}
The inside and outer matrices of the MANOVA matrix~\eqref{defmatrix} are not independent, which is a requirement for our approach. 
However, the eigenvalues of 
\eqref{defmatrix} are the same as those of  
\begin{equation*}
(Y \Ys)^{\oh}(X\Xs+Y\Ys)^{-1}(Y\Ys)^{\oh}=(I+(Y\Ys)^{-\oh}X\Xs(Y\Ys)^{-\oh})^{-1}.
\end{equation*} 
Thus, we begin our work with the matrix $(Y\Ys)^{-\oh}X\Xs(Y\Ys)^{-\oh}$,  which we will call
the {\it product matrix} and 
for which we can  use the approach developed in~\cite{ESY09a,ESY09,ESYY10,PY13} and related papers. 
After determining the behavior of the product matrix we will return to the matrix~\eqref{defmatrix}. 
Our approach determines an implicit equation 
for the Stieltjes transform of the limiting empirical eigenvalue distribution of the product matrix and shows that the solution is stable. 
The implicit equation we obtain is analogous to the quadratic equations that the Stieltjes transforms of the Wigner semicircle law and 
the MP law satisfy. 
Then it is shown that with high probability the Stieltjes transform of the empirical distribution nearly satisfies 
this implicit equation. 
{F}rom stability we conclude how close the empirical distribution is to its limit.

After obtaining results for the matrices  $(Y\Ys)^{-\oh}X\Xs(Y\Ys)^{-\oh}$, 
we return the matrices of our original interest. 
Note that if $\mu$ is an eigenvalue of $(Y\Ys)^{-\oh}X\Xs(Y\Ys)^{-\oh}$, 
then $\lambda=(1+\mu)^{-1}$ is an eigenvalue of~\eqref{defmatrix}. 
The eigenvalues of $(Y\Ys)^{-\oh}X\Xs(Y\Ys)^{-\oh}$ are non-negative  
so that the map just given is regular. 
This allows us to show that if the eigenvalues of $(Y\Ys)^{-\oh}X\Xs(Y\Ys)^{-\oh}$ are
 distributed close to their limiting distribution, 
then the eigenvalues of~\eqref{defmatrix} are also close to their limiting distribution.

\subsection{Conventions}
We make the following conventions, which will be used without referring to them. 
The letters $C$ and $c$ will denote positive absolute constants that may change from appearance to appearance. 
We use the complex number $z=E+i\eta$ for the spectral parameter,  where $\eta$ will always be positive. 
The edges of the limiting spectrum $\lm$ and $\lp$ are given in~\eqref{lpm} 
and we set 
\begin{equation}
\mu_-=\frac{1}{\lp}-1 \;\;\;\textnormal{and}\;\;\;\mu_+=\frac{1}{\lm}-1,\label{defmu}
\end{equation} 
to be limiting spectral edges of the product matrix.

Rather than writing $[a n]$ and $[b n]$ we will write $a n$ and $b n$, and it is implicit that we are 
using the integer part. Most quantities in this paper depend on the
parameters $a$ and $b$, but we usually omit this fact in the notation. 

\section{A Product Matrix}\label{productmatrix}

We begin by addressing the eigenvalues of the  product  matrix $(Y\Ys)^{-\oh}X\Xs(Y\Ys)^{-\oh}$. 
Denote by  $f_p$ 
the limiting empirical spectral distribution for matrices of the form $(Y\Ys)^{-\oh}X\Xs(Y\Ys)^{-\oh}$ 
with Gaussian entries.  
This distribution, which will be derived below, is  explicitly given by 
\begin{eqnarray}
f_{p,a,b}(x)=\fp(x)&=& \frac{C_{a,b}}{2\pi x }\sqrt{\left(\oox-\lm\right)\left(\lp-\oox\right)}
\cdot  I_{[\lm,\lp]}\left(\oox\right), 
\end{eqnarray} 
where $I_{[u,v]}$ is the characteristic function of the interval $[u,v]$. 
We denote this distribution's Stieltjes transform $m_p$, and we will use the subscript $p$ for functions
associated with the product matrix. 
The main theorem of this section, Theorem~\ref{thmprod}, 
 relates the eigenvalues of $(Y\Ys)^{-\oh}X\Xs(Y\Ys)^{-\oh}$ to the functions $f_p$ and $m_p$.
It is followed by 
a delocalization result, Theorem~\ref{deloc}. 
 Our main result, Theorem~\ref{thmmanova}, will also follow from Theorem~\ref{thmprod}.

The following definition formalizes the bulk for the product matrices; 
it is the analogous region to $\EMe$.
\begin{definition} 
For  $\lambda_{\pm}$ and $\mu_{\pm}$ as defined in~\eqref{lpm} and~\eqref{defmu} 
and for $\kappa>0$,  define
\begin{equation*}
\Eek:=\left\{E+i\eta \in\C^+:\;E\in (\mlow,\mhigh)\;\textnormal{ and }
\;\left(\lp-\frac{1}{E+1}\right)\left(\frac{1}{E+1}-\lm\right)\geq \kappa\right\}
\end{equation*}
and set  $\Eck=\mathcal{E}_{\kappa,0}$. 
\end{definition}

The following is the main result on product matrices. 
\begin{theorem}\label{thmprod}
Let $X$ be an $n\times b n$, $b>1$, random matrix with independent entries satisfying~\eqref{subexp} for a uniform $\gamma$. 
Let $Y$ be an $n\times a n$, $a>1$, random matrix independent of $X$ with independent entries 
also satisfying~\eqref{subexp} for the same $\gamma$. 
Let $m_{n,p}(z)$ be the Stieltjes transform of $\YYsh XX^*\YYsh$.  Fix $\kappa>0$ to be a small
positive constant.

\noindent $i)$ 
Then for $\eta> \frac{1}{n\kappa^2}(\log n)^{2C\log\log n}$
\begin{equation}
 \P\left(\sup_{z\in\Eek} |m_{n,p}(z)-\mP(z)|> \frac{(\log n)^{C\log\log n}}{\sqrt{\eta n\kappa}}\right)<n^{-c\log\log n}\label{thmclaimone}
\end{equation}
for all $n$ large enough and for constants $C,c>0$ depending only on $\gamma$. 

\noindent $ii)$ 
Let $\mathcal{N}_{\eta}(E)$ denote the number of 
eigenvalues of $\YYsh\XXs\YYsh$ in $[E-\frac{\eta}{2},E+\frac{\eta}{2}]$, 
and assume  $\eta \ge \frac{1}{n\kappa^2}(\log n)^{3C\log\log n}$.  
Then 
\begin{equation}\label{Nproduct}
\P\left( \sup_{E\in\Eck }\left|\frac{\mathcal{N}_{\eta}(E)}{n\eta}-\fp(E)  \right|\geq 
 \frac{(\log n)^{C\log\log n}}{ (\eta n \kappa)^{1/4} }   \right)\leq n^{-c\log\log n}.
\end{equation}
 \end{theorem}

\begin{theorem}\label{deloc} 
Set $A:=\YYsh XX^*\YYsh$ and assume that $X, Y$ satisfy the same
conditions as in Theorem~\ref{thmprod}. 
Then 
\begin{equation*}
\P\left(\exists \; v\in \C^{n},\;\|v\|_{2}=1,\;Av=\mu v,\;\mu\in\Eck,\;\textnormal{ and }\;\|v\|_{\infty}>\frac{ (\log n)^{C\log\log n}}{\sqrt{ n}}   \right)\leq  n^{-c\log\log n}.
\end{equation*}
\end{theorem}

To prove Theorem~\ref{thmprod} we first determine the
implicit equation for $\mP$. We recall  that the corresponding
implicit equation for both the Wigner semicircle law
and the Marchenko-Pastur law is a simple algebraic equation.
In the case of  the  product matrix,  
 it turns out that $m_p$ satisfies an implicit integral equation  with good stability properties. 
In Section~\ref{SB} we show that  the Stieltjes transform of the empirical density,  $m_{n,p}$, 
approximately satisfies the implicit
equation for $m_p$ and we identify the error term. The error will be controlled in  
 Section~\ref{dj} after several preliminary lemmas in Section~\ref{subevents}.

\subsection{Stieltjes Transform of the Product Matrix}\label{SB}

The first step of the proof is determining an implicit equation for the Stieltjes transform of the  product matrix. 
While determining this equation, we initially view $(Y\Ys)^{-1}$ as a fixed matrix and, adjusting the scaling, 
set $T:=(\on Y\Ys)^{-1}$.

Random covariance matrices of the form $\frac{1}{a n}\YYs$ have limiting distribution
\begin{equation}
f_{MP,a}(x)= f_{MP}(x)=\frac{a}{2\pi}\sqrt{ \frac{\left[\left( (1+a^{-\oh})^2-x\right)\left(x-(1-a^{-\oh})^2 \right)\right]_+ }{x^2}},\label{fMP}
\end{equation}
which is called the Marchenko-Pastur distribution~\cite{MP67}.   
The local Marchenko-Pastur law was obtained in~\cite{ESYY10}, 
and the optimal result for the edge was obtained in~\cite{PY13}. 
The random matrix $T$  then has  limiting distribution
\begin{eqnarray}
f_{\textnormal{Inv},a}(x)=\finv(x)&=& \frac{1}{a x^2}f_{MP}\left(\frac{1}{a x}\right)\nonumber\\
&=&\frac{1}{2\pi x }\sqrt{ \left[\left( (1+a^{-\oh})^2-\frac{1}{a x}\right)\left(\frac{1}{a x}-(1-a^{-\oh})^2 \right)\right]_+ }.\label{finv}
\end{eqnarray}
Note that $\finv$ is supported on $[\am,\aM]$, where $a_{\pm}=(a(1\mp \frac{1}{\sqrt{a}})^2)^{-1}$. 

Here we also address the distribution for $f_p$. 
When the entries of $X$ and $Y$ are Gaussian, the limiting eigenvalue distribution
 of $\YYsh\XXs\YYsh$ can be obtained from the 
MANOVA distribution $f_M$ using the transformation described in Section~\ref{generalapproach}.  
It is
\begin{eqnarray}
f_{p,a,b}(x)=\fp(x)&=&  \frac{C_{a,b}}{(1+x)^2}\frac{\sqrt{\left(\oox-\lm\right)\left(\lp-\oox\right)}}{2\pi \oox\left(1-\oox\right)}\cdot I_{[\lm,\lp]}\left(\oox\right)\nonumber\\
&=&\frac{C_{a,b}}{2\pi x }\sqrt{\left(\oox-\lm\right)\left(\lp-\oox\right)}\cdot I_{[\lm,\lp]}\left(\oox\right).\label{densityp}
\end{eqnarray} 
We use $m_p(z)$ to denote the Stieltjes transform of $f_p$. 
We remark that no explicit formula is available for $\mP$, unlike the case of the semicircle or the 
Marchenko-Pastur law  but it satisfies an integral equation. 
By Lemma 5.1 of~\cite{SB95} the function $\mP$ is the unique solution taking values in $\C^+$
to the implicit equation
\begin{equation}
 m(z)=\int  \frac{1}{\lambda(1-\frac{1}{a}-\frac{1}{a}zm(z))-z}\finv(\lambda)d\lambda.\label{implicitequation}
\end{equation}

For the remainder of Section~\ref{productmatrix} we will shorten the notation $m_{n,p}$ to $m_n$.
Thus, $m_n(z)$ denotes the Stieltjes transform of the product matrix 
$(\YYs)^{-\oh}XX^{*}(\YYs)^{-\oh}=\on T^\oh XX^{*}T^{\oh}$. 
The goal is to derive a self consistent equation for $m_n(z)$ that is close to~\eqref{implicitequation}. 
We will then establish the stability of~\eqref{implicitequation} and conclude that 
$m_n(z)$ is close to $m_p(z)$. 

The derivation of the equation for $m_n(z)$ follows  Silverstein and Bai in~\cite{SB95,Sil95}. 
Here we set $N=b n$.   
Let $X$ be  an $n\times N$  matrix,  and  $T$ an $n\times n$ positive definite matrix. 
Then $\mn $ is the Stieltjes transform of $\frac{1}{n}T^{\oh}XX^*T^{\oh}$ 
and let $\mnu$  be the Stieltjes transform of $\frac{1}{n}X^*TX$. 
Since the nonzero spectrum of $AA^*$ and $A^*A$ coincide for any matrix $A$, we easily get that 
\begin{equation}
\mnu(z)=-\frac{1}{z}\Big( 1-\oa\Big)+\oa  \mn(z).\label{Silv1.2}
\end{equation}

We let $\xj$ denote the $j^{th}$ column of $X$. Set 
\begin{equation*}
\rj:=\frac{1}{\sqrt{N}}T^{1/2}\xj, \qquad
B_n:=\frac{1}{N}T^{1/2}XX^*T^{1/2}=\sum_{j=1}^N \rj \rjs, \qquad
\Bj : = \sum_{i=1,i\neq j}^N \rj \rjs.
\end{equation*}
We let $\mj(z)$ be the Stieltjes transform of $\Bj$ and 
set 
\begin{equation}\label{mnmj}
\mju(z)=-\frac{1}{z}\Big( 1-\oa\Big)+\oa \mj(z).
\end{equation}
We will use the following  simple formula: if $A$ is an $n\times n$ matrix, $q\in \C^n$ and both 
$A$ and $A+qq^*$ are invertible, then 
\begin{equation}
q^*(A+qq^*)^{-1}=\frac{1}{1+q^* A^{-1}q}q^*A^{-1}.\label{Silv2.1}
\end{equation}
Using the definition of $B_n$ we obtain the identity 
\begin{eqnarray*}
I+z(B_n-zI)^{-1}&=&\sum_{j=1}^N \rj \rjs (B_n-zI)^{-1}.
\end{eqnarray*}
Using~\eqref{Silv2.1}, for each $j$
\begin{equation}\label{ident1}
\rjs(B_n-zI)^{-1} = \rjs (\Bj-zI+ \rj\rjs)^{-1}
=\frac{1}{1+ \rjs (\Bj-zI)^{-1}\rj }\rjs (\Bj-zI)^{-1}, 
\end{equation}
so that 
\begin{eqnarray*}
I+z(B_n-zI)^{-1}= \sum_{j=1}^N \frac{1}{1+ \rjs (\Bj-zI)^{-1}\rj } \rj \rjs (\Bj-zI)^{-1}.
\end{eqnarray*}
We take the trace on each side and divide by $N$
\begin{eqnarray*}
\oa +z\oa \mn (z)&=&\frac{1}{N}\sum_{j=1}^N \frac{1}{1+ \rjs (\Bj-zI)^{-1}\rj }\tr[ \rj \rjs (\Bj-zI)^{-1}]\\
&=& \frac{1}{N}\sum_{j=1}^N \frac{\rjs (\Bj-zI)^{-1}\rj }{1+ \rjs (\Bj-zI)^{-1}\rj }\\
&=&1- \frac{1}{N}\sum_{j=1}^N \frac{1 }{1+ \rjs (\Bj-zI)^{-1}\rj },
\end{eqnarray*}
so that 
\begin{equation*}
-\frac{1}{z}\Big(1-\oa\Big)+\oa \mn(z)=-\frac{1}{z} \frac{1}{N}\sum_{j=1}^N \frac{1 }{1+ \rjs (\Bj-zI)^{-1}\rj }.
\end{equation*}
Using~\eqref{Silv1.2} we have 
\begin{equation}
\mnu(z)=-\frac{1}{z} \frac{1}{N}\sum_{j=1}^N \frac{1 }{1+ \rjs (\Bj-zI)^{-1}\rj }.\label{keyident2}
\end{equation}
With the resolvent identity and identities~\eqref{ident1} 
 and~\eqref{keyident2},
\begin{small}\begin{eqnarray}
\lefteqn{(-z\mnu(z)T -zI)^{-1}-(B_n-zI)^{-1}}\nonumber\\
&=&(-z\mnu(z)T -zI)^{-1}\left( \sum_{j=1}^N\rj\rjs-(-z\mnu(z)T)\right)(B_n-zI)^{-1}\nonumber\\
&=&(-z\mnu(z)T -zI)^{-1}\sum_{j=1}^N  \frac{1}{1+ \rjs (\Bj-zI)^{-1}\rj }\rj\rjs (\Bj-zI)^{-1}\nonumber\\
&&-(-z\mnu(z)T -zI)^{-1}\frac{1}{N}\sum_{j=1}^N \frac{1}{1+ \rjs (\Bj-zI)^{-1}\rj }   T(B_n-zI)^{-1}\nonumber\\
&=& \frac{1}{z}\sum_{j=1}^N \frac{1}{1+ \rjs (\Bj-zI)^{-1}\rj }\label{fortrace} \\
&&\times\Big[ (-\mnu(z)T -I)^{-1}\rj\rjs (\Bj-zI)^{-1}-\frac{1}{N} (-\mnu(z)T -I)^{-1}T(B_n-zI)^{-1}\Big].\nonumber
\end{eqnarray}\end{small}
Taking the trace of~\eqref{fortrace} and dividing by $n$ we have 
\begin{equation}
\frac{1}{n}\tr(-z\mnu(z)T -zI)^{-1})-\mn(z)
=\frac{a}{zN}\sum_{j=1}^N \frac{1}{1+ \rjs (\Bj-zI)^{-1}\rj }\dj(z),\label{eqdimn}
\end{equation}
where
\begin{small}\begin{eqnarray*}
\dj(z)
&=&\rjs(\Bj-zI)^{-1}(-\mnu(z)T-I)^{-1}\rj-\frac{1}{a n}\tr(\Bn-zI)^{-1}(-\mnu(z)T -I)^{-1}T.\nonumber
\end{eqnarray*}\end{small}
We break this into four terms
\begin{small}\begin{eqnarray}
|\dj(z)|&\leq&\big|\rjs(\Bj-zI)^{-1}(\mnu(z)T+I)^{-1}\rj-\rjs(\Bj-zI)^{-1}(\mju(z)T+I)^{-1}\rj\big|\label{termA}\\
&+&\; \big|\rjs(\Bj-zI)^{-1}(\mju(z)T+I)^{-1}\rj- \frac{1}{a n}\tr(\Bj-zI)^{-1}(\mju(z)T +I)^{-1}T\big|\label{termB}\\
&+&\;\Big|\frac{1}{a n}\tr(\Bj-zI)^{-1}(\mju(z)T +I)^{-1}T-\frac{1}{a n}\tr(\Bj-zI)^{-1}(\mnu(z)T +I)^{-1}T \Big|  \label{termC}\\
&+&\;\Big| \frac{1}{a n}\tr(\Bj-zI)^{-1}(\mnu(z)T +I)^{-1}T -\frac{1}{a n}\tr(\Bn-zI)^{-1}(\mnu(z)T +I)^{-1}T \Big|.\label{termD}
\end{eqnarray}\end{small}
We will bound $|d_j(z)|$ in Section~\ref{dj} using the Lemmas developed in Sections~\ref{subevents} and~\ref{inteq}.

\begin{subsection}{Large deviation estimates for exceptional events}\label{subevents}
In this section we will define certain typical events, 
denoted by $\X,\To,\Tt$ and $\B$, that will be needed to estimate $|d_j(z)|$. 
Recall that $\xj$ denotes the $j^{th}$ column of $X$ and let $\Xj$  denote the matrix 
obtained by removing the $j^{th}$ column of $X$. Define the event 
\begin{equation}
\X:= \Big\{ \oh\leq \frac{\|\xj\|^2_{2}}{ n}\leq 2\;\;\textnormal{and}\;\;
\left(1-\frac{1}{\sqrt{b}}\right)^2\leq\left\|\on\Xj\Xj^*\right\|<4b\; 
  \;\textnormal{ for all}\;\;1\leq j\leq  b n \Big\}.\label{eqX}
\end{equation}
For some small $c>0$, let $\To$ denote the event
\begin{equation}
\To: =\Big\{ T=\onYYsi\;\;\textnormal{ is well-defined and }\;\;   (1-c)\am   \leq T\leq (1+c)\aM \Big\}.\label{eqTone}
\end{equation}
Let $\{t_1,\ldots,t_n\}$ denote the eigenvalues of $T$.
Let $\{\tau_1,\ldots,\tau_n\}$ denote  their classical locations given 
through the limiting density $\finv$ from~\eqref{finv}, i.e. they are 
defined through the formula 
\begin{equation}
\int_{-\infty}^{\tau_k}\finv(x)dx=\frac{k}{n}, \qquad k=1,\ldots,n.\label{deftau}
\end{equation}
Let 
\begin{equation}
\Tt:=\Big\{ \on\sum_{k=1}^n |t_k-\tau_k|\leq \frac{1}{n}(\log n)^{C_T\log\log n}  \Big\},\label{eqTtwo}
\end{equation}
 denote the event that the actual eigenvalues are close to their classical location,
where $C_T>0$ is a  constant independent of $n$.

Let $\mu^{(j)}_1,\ldots,\mu^{(j)}_{n-1}$ denote the eigenvalues of $\Bj$ for $j=1,\ldots,  N$. 
For a fixed constant $K$, let  $\B$ be the event that 
\begin{equation}
\B:=\Big\{ \sup_{z\in {\cal E}^{(\mu)}}\max_{1\leq j\leq  N  } \on \sum_{k=1}^{n-1}\frac{1}{|\mu_k^{(j)}-z|}\leq K (\log n)^2\Big\}
  \cap \Big\{ \sup_{z\in {\cal E}^{(\mu)}}\max_{1\leq j\leq  N } \on \sum_{k=1}^{n-1}\frac{1}{|\mu_k^{(j)}-z|^2}\leq \frac{K}{\eta}
 \Big\},
\label{eqB}
\end{equation}
where 
\begin{small}\begin{equation*}
{\cal E}^{(\mu)}
:= \Big\{ z=E+i\eta\in \C\; : E\in(\mlow,\mhigh),\;
 \left(\lp -\frac{1}{1+E}\right)\left(\frac{1}{1+E}-\lm\right)\geq \kappa,\; \frac{(\log n)^{C}}{n}\leq \eta\leq 1
\Big\}.
\end{equation*}\end{small}

\begin{proposition}\label{events}
 With the notations above, the following estimate holds
\begin{equation}\label{fourset}
\P(\X\cap\To\cap\Tt\cap\B)\geq 1-n^{-c\log\log n}.
\end{equation}
\end{proposition}

We first prove several lemmas.

\begin{lemma}\label{ttau}
Let $\{\otau_1,\ldots,\otau_n\}$ be the classical locations w.r.t. the Marchenko-Pastur law, i.e.
\begin{equation}
\int_{-\infty}^{\otau_k}f_{MP}(x)dx=\frac{k}{n}, \qquad k=1,2,\ldots, n.\label{properplacement}
\end{equation}
Assume that the points   $\{\ot_1,\ldots,\ot_n\}$ satisfy 
\begin{equation}\label{sat}
|\ot_k-\otau_k|\leq \delta, \qquad \min(\ot_1,\ldots,\ot_n)\geq \alpha >0
\end{equation}
 for some $\delta, \alpha>0$. 
Set  $t_k:=(a \ot_{n+1-k})^{-1}$.
Then,  recalling the definition of $\tau_k$ from~\eqref{deftau}, we have
$$
|\tau_k-t_k|\leq \frac{\delta}{\alpha a \left(1-a^{-\oh}\right)^2}, \qquad  k=1,2,\ldots, n.
$$
\end{lemma}

\begin{proof}
The smallest point contained in the support of $f_{MP}$ is $(1-a^{-\oh})^2$, 
so given~\eqref{properplacement} we have
\begin{equation*}
\int_{(a\otau_k)^{-1}}^{a^{-1}(1-a^{-\oh})^{-2}} \frac{1}{a x^2} f_{MP}\left(\frac{1}{a x}\right)dx=\frac{k}{n}.
\end{equation*}
Considering the relation between $\finv$ and  $f_{MP}$ (see~\eqref{finv}),  
 it follows that $\tau_k:=(a \otau_{n+1-k})^{-1}$. 
By the assumption \eqref{sat} we have 
\begin{equation*}
\Big|\frac{1}{a\otau_k}-\frac{1}{a\ot_k}\Big|\leq \frac{\delta}{\alpha a(1-a^{-\oh})^2} 
\end{equation*}
for all $k=1,2, \ldots n$. 
\end{proof}

The following lemma is a variation of Lemma 4.7 of~\cite{ESY10}. 
\begin{lemma}\label{like4.7ofWegner}
Let the entries of $b\in\C^{n}$ be uniformly $\gamma$-subexponential independent random variables satisfying $\E b_i=0$ 
and $\E|b_i|^2=1$ for $1\leq i\leq n$.  
Let $S$ be an $n\times n$ positive definite matrix satisfying $\sm\leq S\leq \sM$ for $0<\sm\le 1$ and $1\le \sM<\infty$. 
Let $\{v_i\}_{i=1}^m$ be a set of orthonormal vectors in $\C^n$ and set $\xi_i=|b^* Sv_i|^2$ for $i=1,\ldots,n$. 
If $\epsilon >0$ and $m$ satisfy  $ \oh(1-\epsilon)\sm^2\sqrt{m}\geq \sM^2(\log n)^{C} $
for a constant $C$ depending only on $\gamma$,  then 
\begin{equation*}
\P\left(\sum_{i=1}^m \xi_i\leq \epsilon \sm^2 m\right)\leq Cn^{-\log\log n}. 
\end{equation*}
\end{lemma}

\begin{proof}
We set $w_i=Sv_i$ for $1\leq i\leq m$. 
As in the proof of Lemma 4.7 in~\cite{ESY10}, we set 
\begin{equation*}
Z= \sum_{k,l=1}^n a_{k,l}[b_k\overline{b}_l-\E b_k\overline{b}_l]\;\;\textnormal{for}\;\;\;a_{k,l}=\sum_{i=1}^m\overline{w}_i(k)w_i(l).
\end{equation*}
We note that $\sum_{i=1}^m\xi_i=Z+\sum_{i=1}^m\|w_i\|^2$. 
Also,  $\sum_{k,l=1}^n|a_{k,l}|^2=\tr \big[\sum_{i=1}^m(Sv_i)(v_i^*S)\big]^2\le \sM^4m$, 
where in the last step we used that $M=\sum_{i=1}^m(Sv_i)(v_i S^*)$ is a matrix of rank at most $m$ and 
$M\le S^2\le \sM^2$. 
Our assumptions on $\epsilon$ and $m$ give $\oh(1-\epsilon)\sm^2 m\geq (\log n)^{C}\sM^2\sqrt{m}$. 
Therefore, using Lemma B.2 of~\cite{EYY11}, 
\begin{eqnarray*}
\P\Big(\sum_{i=1}^m\xi_i\leq \epsilon\sm^2 m\Big)
&\le &\P\Big(|Z|>\sum_{i=1}^m\|w_i\|^2-\epsilon\sm^2 m\Big)\\
&\leq &\P\Big(|Z|>(1-\epsilon)\sm^2 m\Big)\\
&\leq & \P\left(  \left|\sum_{k\neq l} a_{k,l}[b_k\overline{b}_l-\E b_k\overline{b}_l ]\right|>   (\log n)^{C}\sM^2\sqrt{m}  \right)\\
&&+\P\left(\left| \sum_{k=1}^n a_{k,k}[b_k\overline{b}_k-\E b_k\overline{b}_k] \right|>   (\log n)^{C}\sM^2\sqrt{m}  \right)\\
&\leq &  \P\left(  \left|\sum_{k\neq l} a_{k,l}[b_k\overline{b}_l-\E b_k\overline{b}_l ]\right|> (\log n)^{C}\left( \sum_{k\neq l} |a_{k,l}|^2\right)^{\oh}\right)\\
&&+\P\left(\left| \sum_{k=1}^n a_{k,k}[b_k\overline{b}_k-\E b_k\overline{b}_k] \right|> (\log n)^{C}\left(\sum_{k=1}^n |a_{k,k}|^2\right)^\oh\right)\\
&\leq &C n^{-\log\log n}. 
\end{eqnarray*}
\end{proof}

\begin{lemma}\label{bins}
Assume $E\in\Eck$, set $\Ieta=[E-\frac{\eta}{2},E+\frac{\eta}{2}]$ and let $\N_{\eta}$ 
denote the number of eigenvalues of $\YYsh\XXs\YYsh$ in $\Ieta$. 
If $\frac{(\log n)^{C}}{ n}<\eta<E/2$  then 
there exist  constants $c,K>0$, 
depending only on $\gamma$ such that
\begin{equation*}
\P\left(\N_{\eta}  \geq \frac{Kn\eta}{\sqrt{E}} \right)\leq n^{-c\log\log n}
\end{equation*}
for all large $n$. 
\end{lemma}
\begin{proof}
First, we use that $\YYsh\XXs\YYsh$ and $\Xs(\YYs)^{-1}X=\on \Xs T X $ have the same nonzero eigenvalues 
to justify working with the latter matrix. 
Now, following the proof of Lemma~8.1 in~\cite{ESYY10}, we need to bound the absolute value of the diagonal entries of
 $(\on \Xs T X-zI)^{-1}$.  
We consider the $(1,1)$ entry and  
let $x_1$ denote the first column of $X$ and $\Xo$  the matrix obtained by removing this column from $X$. 
We use the following identity for an arbitrary matrix $A$, which can be seen by using a singular value decomposition, 
\begin{equation*}
A(A^*A-zI)^{-1}A^*=I+z(AA^*-zI)^{-1}. 
\end{equation*}
We will use this identity for the matrix $A=\frac{1}{\sqrt{n}}T{\oh}\Xo$. 
By the matrix inversion formula, 
\begin{eqnarray*}
\left(\on \Xs TX-zI\right)^{-1}(1,1)&=&\frac{1}{ \frac{1}{n}x_1^* T x_1-z-\frac{1}{n^2}x_1^*T\Xo(\on\Xo^*T\Xo -zI)^{-1}\Xo^*T x_1 }\\
&=& \frac{1}{ \on x_1^* T x_1-z-\on x_1^*T^{\oh}(I+z(\frac{1}{n}T^{\oh}\Xo\Xo^* T^{\oh} -zI)^{-1})T^{\oh} x_1 }\\
&=&\frac{1}{-z-\on zx_1^*T^{\oh} (\on T^{\oh}\Xo\Xo^* T^{\oh}-zI)^{-1}T^{\oh} x_1 }. 
\end{eqnarray*}
Let $\mu_1,\ldots,\mu_{n-1}$ and $u_1,\ldots,u_{n-1}$ denote the eigenvalues and normalized eigenvectors of $\on T^{\oh}\Xo\Xo^*T^{\oh}$. 
Then setting $\xi_k=|u_k^* T^{\oh}x_1 |^2$ for $1\leq k\leq n-1$, we have 
\begin{eqnarray*}
\left| \left(\on \Xs T^{-1}X-zI\right)^{-1}(1,1)  \right| &=&\frac{1}{|z(1+ \frac{1}{n}\sum_{k=1}^{n-1}\frac{\xi_k}{\mu_k-z}  )|} \\
&\leq & \frac{n}{|z|\Im \sum_{k=1}^{n-1}\frac{\xi_k}{\mu_k-z}  } \\
&\leq&  \frac{n}{E \;\Im \sum_{k:|\mu_k-E|\leq \frac{\eta}{2}  }\frac{\xi_k}{\mu_k-z}}\\
&\leq &  \frac{Cn\eta}{E\sum_{k:|\mu_k-E|\leq \frac{\eta}{2}}\xi_k}.
\end{eqnarray*}
Continuing to follow~\cite{ESYY10}, we have 
\begin{equation}
\P\left(\N_{\eta}  \geq \frac{Kn\eta}{\sqrt{E}} \right)
\leq n\P\left( \sum_{k:\mk\in I_\eta}\xi_k\leq  \frac{Cn\eta}{K\sqrt{E}}\;\;\& \;\;\N_{\eta}\geq \frac{Kn\eta}{\sqrt{E}} \right). \label{probconditioned}
\end{equation}
By Theorem 3.1 in~\cite{PY13}, inequality~\eqref{eqTone} occurs with probability at least $1-Cn^{-\log \log n}$, 
i.e. $\oh\am\leq T\leq 2\aM$. 
By setting $\epsilon=\oh$ in Lemma~\ref{like4.7ofWegner} and choosing $K$ large enough, we satisfy 
\begin{equation*}
\oh\left(\oh\am\right)^2\sqrt{\N_{\eta}}\geq\oh\left(\oh\am\right)^2 \sqrt{\frac{Kn\eta}{\sqrt{E}}}\geq \oh\left(\oh\am\right)^2\sqrt{\frac{K}{\sqrt{E}}(\log n)^{C}}\geq 2\aM^2(\log n)^{C}.
\end{equation*}
We now apply Lemma~\ref{like4.7ofWegner} to obtain the claim. 
\end{proof}

\begin{lemma}\label{boundsB} 
Let $\mk^{(j)}$, $k=1,\ldots,n-1$ denote the eigenvalues of $\Bj$  for an arbitrary  $j=1,2,\ldots N$,   
and assume that $\eta>n^{-1}(\log n)^{C}$ and that $E$ satisfies  $(E-\lm)(\lp-E)\geq\kappa$.
Then, when $\To$ holds,  there exists a positive constant $K$ such that 
with probability at least $1-n^{-c\log\log n}$ 
\begin{equation}
\sup_{z\in{\cal E}^{(\mu)}}\max_{1\leq j\leq  N }\left\{ \on \sum_k \frac{1}{|\mk^{(j)}-z|}\right\}\leq K (\log n)^2\label{boundtrace}
\end{equation}
and
\begin{equation}
\sup_{z\in{\cal E}^{(\mu)}} \max_{1\leq j\leq  N  }\left\{\on 
\sum_k \frac{1}{|\mk^{(j)}-z|^2}\right\}\leq \frac{ K}{\eta}.\label{boundHS}
\end{equation}
\end{lemma}
\begin{proof}
For a fixed $z\in {\cal E}^{(\mu)}$ and index $j$, 
given the bound on $\P\left(\N_{\eta}  \geq \frac{Kn\eta}{\sqrt{E}} \right)$ from Lemma~\ref{bins}, 
the claim follows from the same calculation as is done in the proof of Proposition~4.3 in~\cite{ESY10}.
This proves the estimate for each fixed $z$. To obtain the result simultaneously for all $z\in {\cal E}^{(\mu)}$, we 
notice that the  derivatives of the functions to be bounded in~\eqref{boundtrace} and~\eqref{boundHS} are 
bounded by $Cn\eta^{-4}$ on $\mathcal{E}^{(\mu)}$.  
Thus, we may discritize $\mathcal{E}^{(\mu)}$ to $Cn^2\eta^{-8}$ points and take the union bound with respect to the 
discrete grid and  the indices $1\leq j\leq  N  $. 
\end{proof}

\begin{proof}{\bf of Proposition~\ref{events}}
We will prove the following four inequalities.
\begin{eqnarray}
\P(\X)&\geq&1-n^{-c\log\log n}\label{aa}\\
\P(\To)&\geq & 1-n^{-c\log\log n}\label{bb}\\
\P(\Tt|\To)&\geq & 1-n^{-c\log\log n}\label{cc}\\
\P(\B|\To)&\geq & 1-n^{-c\log\log n}.\label{dd}
\end{eqnarray}

Let $x$ denote an arbitrary column of $X$. 
We assume that $n$ is large enough so that 
$(\log n)^{C}<n^{\oh}$ and  apply Lemma B.2 of \cite{EYY11} to the identity matrix.  
Then~\eqref{aa} follows by summing the  $N=b n$   probabilities.
Inequality~\eqref{bb}
follows from the much stronger statement in Theorem~3.1 of~\cite{PY13}.  

Recall the definition of  $\otau_1,\ldots,\otau_n$ from~\eqref{properplacement} 
and let $\ot_1,\ldots,\ot_n$ denote the eigenvalues of $\frac{1}{a n}XX^{*}$. 
Then 
Theorem~3.3 of~\cite{PY13} gives 
\begin{equation*}
|\otau_k-\ot_k|\leq  (\log n)^{C_T\log\log n} n^{-2/3}(\min(n+1-k,k))^{-1/3}\;\;\textnormal{for all}\;k =1,2,\ldots n,
\end{equation*}
with probability at least $1-C\exp\left(-(\log n)^{c\log\log n}\right)$.  
Lemma~\ref{ttau} then implies that 
\begin{equation*}
|\tau_k-t_k|\leq  (\log n)^{C_T\log\log n} n^{-2/3}(\min(n+1-k,k))^{-1/3}\;\;\textnormal{for all}\;k
\end{equation*}
with the same probability. 
Thus, assuming $\To$ holds to address $t_1$ and $t_n$, we have shown~\eqref{cc}. 
Lastly, Lemma~\ref{boundsB} yields~\eqref{dd}.
Together these give the bound
\begin{equation*}
\P(\X\cap\To \cap\Tt\cap\B)\geq  1-n^{-c\log \log n}.
\end{equation*}
\end{proof}
\end{subsection}

\subsection{Integral Equation}\label{inteq}

In this section  we show that the integral equation~\eqref{implicitequation} that $m_p(z)$ satisfies is stable.
This means that if $\md$ satisfies 
\begin{equation}
 \md(z)=\int \frac{1}{\lambda(1-\frac{1}{a}-\frac{1}{a}z\md(z))-z}\finv(\lambda)d\lambda+\delta(z),\label{perta}
\end{equation}
for some small $\delta(z)$, 
then $\md$ is close to $\mP$.

\begin{lemma}\label{stability}
Assume $\md(z)\in \Cp$ is analytic on the upper half plane  and is a solution to the perturbed equation
\eqref{perta}. Let $0<\eta'< 1$, $\kappa>0$ and assume $\To$ holds. Let  $E$ be  chosen so that 
$\left(\lp-\frac{1}{E+1}\right)\left(\frac{1}{E+1}-\lm\right)\geq \kappa$.
There exists a small universal constant $c_1>0$ 
such that if 
\begin{equation}
|\mud(E+i\eta)-\munp(E+i\eta)|\leq c_1\sqrt{\kappa+\eta}\label{constdistance}
\end{equation} 
for all $\eta\in [\eta',1]$ and 
\begin{equation}
\sup_{\eta\in[\eta',1]}|\delta(E+i\eta)|<\delta_1
\end{equation}
for some $\delta_1 \le c_1\sqrt{\kappa} $,   then  
\begin{equation}
|\mud(E+i\eta)-\munp(E+i\eta)|\leq C_2\frac{\delta_1}{\sqrt{\kappa+\eta}}\label{claima}
\end{equation}
and 
\begin{equation}
|\md(E+i\eta)-\mP(E+i\eta)|\leq C_2\frac{\delta_1}{\sqrt{\kappa+\eta}},\label{claimb}
\end{equation}
for all $\eta\in [\eta',1]$, 
where $C_2$ depends only on $c_1$. 
\end{lemma}
We define 
\begin{equation}
\munp(z)=-\frac{1}{z}\Big(1-\oa\Big)
+\oa \mP(z)\;\;\textnormal{   and   }\;\; 
\mud(z)=-\frac{1}{z}\Big(1-\oa\Big)+\oa \md(z).\label{under}
\end{equation}  
Notice that for $\delta=0$, $\md$ and $\mud$ are the same as $\mP$ and $\munp$.
Simple algebra shows that from~\eqref{perta} we have the following equation for $\mud$:
\begin{equation}
\mud(z)\left(z-\oa \int \frac{\lambda}{1+\lambda\mud(z)}\fdl   \right)=-1+\frac{z\delta}{a}.     \label{pert2}
\end{equation}

We will work mostly with $\mud$ instead of $\md$. 
We introduce the notation  
\begin{equation*}
\Kz=\oa\int A^2(\lambda) \fdl\;\;\textnormal{and}\;\;  A(\lambda)=  \frac{\lambda \munp(z)}{1+\lambda\munp(z)}. 
\end{equation*}
Note that both $\Kz$ and $A$ depend on $z$, which we omit  writing here but will include 
in the proof of Lemma~\ref{stability}. 
The following lemma bounds $\Kz$ away from $1$,  which will be used in the proof of Lemma~\ref{stability}. 
We set 
\begin{equation*}
{\cal G}:=\Big\{z=E+i\eta\in \C:\;\;E\in(\mlow,\mhigh),\;\left(\lp-\frac{1}{E+1}\right)
\left(\frac{1}{E+1}-\lm\right)\geq \kappa,\;0 < \eta\leq 1\Big\}. 
\end{equation*}
\begin{lemma}\label{Kzero}   
There exist two positive constants $C$ and $c$ 
 such that the 
 following statements hold uniformly  for all $z=E+i\eta\in{\cal G}$. 
\begin{enumerate}
\item  
\begin{equation}
c\leq |\munp(z)|\leq C,   \qquad \mbox{and}\qquad c\leq |\mP(z)|\leq C,  \label{range} 
\end{equation}
and $\munp(z)$, $\mP(z)$ extend continuously to the interval  $[\mlow,\mhigh]\subset \R$.  
\item 
\begin{equation}
\Im \munp(z)\geq \oa\Im \mP(z)\geq c\sqrt{\kappa +\eta }.\label{lowermp}
\end{equation}
\item Let $\mun(z)$ be a function $\C^+\rightarrow \C^+$. 
Suppose for some $G\subset \G$
\begin{equation}\label{mmm}
\sup_{z\in G}|\munp(z)-\munder(z)|<c.
\end{equation}
When $\To$ holds 
with a sufficiently small $c$  in its definition \eqref{eqTone}, then 
\begin{equation}
\sup_{z\in G}\| (\munder(z)T+I)^{-1}\|<C.\label{mtplusI}
\end{equation}
\item 
\begin{equation}
|1-\Kz(z)|>c\sqrt{\kappa  + \eta }.\label{distKz}
\end{equation}
\end{enumerate}
\end{lemma}

\begin{proof}
1.) From 
$\eqref{pert2}$ with $\delta=0$ we have the following equation for $\munp(z)$
\begin{equation}
-\frac{1}{\munp(z)}=z-\oa\int \frac{\lambda}{1+\lambda \munp(z)}\fdl. \label{eqmp}
\end{equation}
If $\{z_k\}_{k=1}^\infty$ is   sequence in $\{E+i\eta:\;E\in (\mlow,\mhigh),\;0< \eta\leq 1\}$ 
such that either $\lim_{k\rightarrow\infty}|\munp(z_k)|=0$ or $\lim_{k\rightarrow\infty}|\munp(z_k)|=\infty$, 
then  since $0<\mlow \le |z_k|\le \mhigh+1$ for all $k$,~\eqref{eqmp} would be violated. 
Thus, no such sequence exists, which proves~\eqref{range}. 
The continuous extension follows from $\munp$ being analytic on $\C^+$.  
The statements for $\mP$ follow from \eqref{under} and the fact that $E=\Re z$ is separated
away from zero. 

2.)
We  define the  interval $J:=[\am+c(\kappa +\eta ),\aM-c(\kappa+ \eta )]$; 
by setting the constant $c$ sufficiently small and using $\eta\le 1$ we
can ensure that $|J|\ge \frac{1}{2}(\aM-\am)$. 
Since $\fp$ has a square root singularity near the edge of its support, 
we have $\fp(\lambda)\geq c\sqrt{\kappa +\eta }$ for all $\lambda\in J$. Thus 
\begin{eqnarray}
\Im \mP(z)&=&\int\frac{\eta}{(E-\lambda)^2+\eta^2}f_p(\lambda)d\lambda    \nonumber\\
&\geq&c\sqrt{\kappa + \eta }\int_{J}\frac{\eta}{(E-\lambda)^2+\eta^2}d\lambda\nonumber\\
&\geq& c\sqrt{\kappa+ \eta },
\end{eqnarray} 
where the last inequality follows from the fact that the length of $J$  and the distance from 
$E$ to $J$ are both  $O(1)$ and $\eta\leq 1$. 
We then also have  
$\Im \munp(z)\geq \oa\Im \mP(z)\geq c_1\sqrt{\kappa+ \eta}$
maybe with a smaller constant $c_1$.

3.) 
We take the imaginary part of~\eqref{eqmp} to obtain 
\begin{equation}\label{optical}
\frac{\Im \munp}{|\munp|^2}=\eta+\oa\int  \frac{\lambda^2\Im\munp}{|1+\lambda \munp|^2}\fdl,
\end{equation}
so that 
\begin{equation}
1=\frac{\eta |\munp|^2}{\Im \munp}+\oa\int |A(\lambda)|^2\fdl\label{oneeq}
\end{equation}
and 
\begin{equation}
\oa\int |A(\lambda)|^2\fdl\leq 1.\label{mpCS}
\end{equation}
{F}rom~\eqref{mpCS}, 
 the uniform bound on $|\munp(z)|$ and the lower bound on the support of 
$\finv$ we have 
\begin{equation}
\int \frac{1}{|1+\lambda \munp(z)|^2}\fdl<C'\label{boundint}
\end{equation}
with some constant $C'$ uniformly for all $z\in\G$.
Suppose that 
\begin{equation*}
w(z):=\inf_{\lambda\in[\am,\aM]}|1+\lambda\munp(z)|
\end{equation*}
is attained at $\lambda=\lambda_0(z)$ for any $z\in \G$. 
Note that $w(z)\ge \am \Im \munp(z)\ge c\sqrt{\kappa+\eta}$ with some positive constant $c$,
where we used \eqref{lowermp}.
Since the derivative of $\lambda\mapsto |1+\lambda\munp(z)|$
is uniformly bounded by~\eqref{range},
there exists a subinterval $J= J(z)\subset[\am,\aM]$ of length at least $cw(z)$
 such that $|1+\lambda\mun(z)|\leq 2w(z)$ 
for all $\lambda\in J$. 
Since $\finv(\lambda)$ has a square-root singularity at its edges and is bounded away from zero between the 
edges, we have 
\begin{equation}
C'\geq \int \frac{1}{|1+\lambda \munp(z)|^2}\fdl\geq \frac{1}{4 w(z)^2}\int_J \fdl\geq \frac{c}{\sqrt{w(z)}}.
\end{equation} 
Therefore we have a uniform lower bound  
\begin{equation}\label{uniflow}
|1+\lambda\munp(z)|>c',
\end{equation}
for $c'$ depending only on  the constant $C'$ in \eqref{boundint},   
for all $\lambda\in (\am,\aM)$ and for all $z\in \G$.  
Using continuity in both $\lambda$ and $\munp(z)$
and using that \eqref{mmm} holds with a sufficiently small $c$,
 we have $|1+t\munder(z)|>c$ for all 
$t\in[(1-c)\am,(1+c)\aM]$ and $z\in G$ if $c$ is chosen sufficiently small.
By applying the spectral theorem for $T$ and \eqref{eqTone} with a small $c$, we obtain
\eqref{mtplusI}.

4.) The variable $z$ plays no role in the remainder of the proof and so we omit it from the notation. 
By the assumption on $E$, we have $\Im \munp\geq c\sqrt{\kappa  +\eta }$. 
The property $|\munp(z)|<C$  and $\lambda\leq \aM$ give
\begin{equation}
\frac{\Im A(\lambda)}{\left|A(\lambda)\right|}
=\frac{  \Im  \munp}{|1+\lambda \munp||\munp|}
\geq  c\Im \munp(z) \ge c\sqrt{\kappa  +\eta } .\label{lowerim}
\end{equation}
We set 
\begin{equation*}
B:=\oa\int |A(\lambda)|^2\fdl, 
\end{equation*}
so that by~\eqref{mpCS} we have   $|\Kz|\leq B<1$. 
We claim that  
\begin{equation}
|B-\Kz|\geq cB \sqrt{\kappa +\eta} \label{distB}
\end{equation}
 for a positive constant $c$. 

By the lower bound on $|\munp(z)|$ in~\eqref{range} 
 and on $|1+\lambda\munp(z)|$ in~\eqref{uniflow},   
there exists a constant $C$  such that 
\begin{equation}
\frac{\Im A(\lambda)}{|A(\lambda)|}=\frac{\Im \munp(z)}{|1+\lambda\munp(z)||\munp(z)|}\leq C\Im \munp(z)\label{boundsin}
\end{equation} 
for all $z\in \G$ and $\lambda\in(\am,\aM)$.
Let $\epsilon>0$ be  a sufficiently small constant so that $1- 2C\epsilon\geq \oh$
with the constant $C$ from \eqref{boundsin}. 
If $\Im \munp(z)>\epsilon$, then, using~\eqref{lowerim},
\begin{eqnarray*}
|B-\Kz|&=&
\oa\left| \int\left(| A(\lambda)|^2-  A^2(\lambda)\right)\fdl   \right|\\
&=&\oa\left| \int 2(\Im A(\lambda))^2\fdl -2i\int (\Re A(\lambda))(\Im A(\lambda))\fdl\right|\\
&\geq &\frac{2}{a} \int \left(\Im A(\lambda) \right)^2 \fdl   \\
&\geq & c\epsilon\Im \munp(z)\int |A(\lambda)|^2\fdl \\
&\geq & c\epsilon B\sqrt{\kappa+\eta}.
\end{eqnarray*}

If $\Im \munp(z)\leq \epsilon$, then we set $A(\lambda)=e^{i\phi(\lambda)}|A(\lambda)|$. 
We note that $\phi\in(0,\pi)$ since $\Im A(\lambda)>0$, and that $\phi$ is well-defined since 
$|A|\neq 0$. 
By~\eqref{boundsin} we have $\sin \phi\leq C\epsilon$, 
and by continuity, either $0\leq \phi\leq C\epsilon$ or $\pi-C\epsilon\leq \phi\leq \pi$. 
In both cases we have $\cos\phi\geq \oh$ if $\epsilon$ is small. 
So we have 
\begin{eqnarray*}
|B-\Kz|&=&
\oa\left| \int\left(| A(\lambda)|^2-  A^2(\lambda)\right)\fdl   \right|\\
&=&\frac{2}{a}\left| \int |A(\lambda)|^2e^{i\phi}\sin\phi\fdl\right| \\
&\geq &\frac{2}{a}\left| \int |A(\lambda)|^2  \cos \phi \sin \phi\fdl \right|
-\frac{2}{a}\int |A(\lambda)|^2\sin^2\phi \fdl \\
&\geq & \oa \int |A(\lambda)|^2 (\sin \phi-  2  \sin^2\phi)\fdl\\
&\geq& \frac{1}{2a}  \int |A(\lambda)|^2 \sin \phi\fdl\\
&\geq& c\Im \munp(z)\int |A(\lambda)|^2\fdl,
\end{eqnarray*}
where for the last inequality we use   $\sin \phi \sim \Im A/|A|$ and~\eqref{boundsin}.  
Since $\epsilon$ depends only on $C$  from~\eqref{boundsin}, we use~\eqref{lowermp} to obtain
\begin{equation*}
|B-\Kz|\geq cB\sqrt{\kappa +\eta}.
\end{equation*}
We trivially have $|1-\Kz|\geq |B-\Kz|$ for any positive number $B$ with $B\leq 1$ and complex number $\Kz$ with 
$|\Kz|\leq B$. 
Therefore~\eqref{distKz} follows from~\eqref{distB} and from a bound $B\geq c>0$. 
The bound $B\geq c>0$, follows  from the fact that $|A(\lambda)|\ge c>0$ on
the support of $\finv$ and 
  that $\finv$ is a probability measure 
whose support is separated from zero. 
\end{proof}

\begin{proof}{\bf of Lemma~\ref{stability}}
Throughout the proof, $\delta$ is a function of $z$, but we will omit writing this dependence. 
We also fix $E=\Re z$ and we vary only $\eta = \Im z$. 
Note that from definition~\eqref{under} we have 
\begin{equation}
\mP(z)-\md(z)=a(\munp(z)-\mud(z)), \label{difforiginal}
\end{equation}
so that~\eqref{claimb} follows from~\eqref{claima}.

 To prove~\eqref{claima} we observe that 
\begin{equation}\label{selfcons}
\munp(z)-\mud(z)=(\munp(z)-\mud(z))K(z)- \frac{z}{a}\left(z-\oa\int \frac{\lambda}{1+\lambda\mud(z)}\fdl\right)^{-1}   \delta,
\end{equation}
where 
\begin{equation}
 K(z):=\frac{ \oa\int \hdl\hpl\finv(\lambda) d\lambda}
{ \left( z-\oa\int \hdl \finv(\lambda)d\lambda\right)\left( z-\oa\int \hpl\finv(\lambda) d\lambda\right)}\label{defKz}
\end{equation}
is obtained by taking the difference of the expressions for $\munp(z)$ and $\mud(z)$ given in~\eqref{pert2} 
and~\eqref{eqmp}.  
Since $|\mud(z)|\leq C$ and, for $\delta$ small enough, $|z\delta|\leq \frac{a}{2}$, using~\eqref{pert2} we have 
$$
\left|\left(z-\oa\int \frac{\lambda}{1+\lambda\mud(z)}\fdl\right)^{-1}  \right|\leq C,
$$ 
and therefore \eqref{selfcons} yields 
\begin{eqnarray}
 |\munp(z)-\mud(z)|= \frac{C|\delta|}{|1-K(z)|}.\label{diffK}
\end{eqnarray}
{F}rom~\eqref{pert2} and~\eqref{defKz} we have
\begin{equation}
K(z)=\oa\int \frac{\lambda\munp(z)}{1+\lambda\munp(z)}\frac{\lambda\mud(z)}{1+\lambda\mud(z)}\fdl
-\frac{z\delta}{z\delta-a}\int \frac{\lambda\munp(z)}{1+\lambda\munp(z)}\frac{\lambda\mud(z)}{1+\lambda\mud(z)}\fdl.\label{Kz2}
\end{equation}
Since $|z\delta|\leq \frac{a}{2}$, $|\delta|\le\delta_1$,
 the absolute value of the second term in~\eqref{Kz2} is bounded by $C\delta_1$. 
Here we used \eqref{uniflow} and  that a similar positive lower bound holds for $|1+ \lambda\mud(z)|$ 
as well, assuming that $c_1$ in \eqref{constdistance} is sufficiently small. 
Thus,
\begin{eqnarray*}
|K(z)-\Kz(z)|
& \leq &\left|\oa\int \frac{\lambda\munp(z)}{1+\lambda\munp(z)}
\left(\frac{\lambda\mud(z)}{1+\lambda \mud(z)}- \frac{\lambda\munp(z)}{1+\lambda \munp(z)}\right) \fdl   \right|+C\delta_1 \\
&\le&  \oa\int \left| \frac{\lambda\munp(z)}{1+\lambda\munp(z)}
  \frac{\lambda(\mud(z)-\munp(z))}{(1+\lambda\munp(z))(1+\lambda\mud(z))} \right| \fdl+C\delta_1\\
&\leq& C |\mud(z)-\munp(z)|+C\delta_1.
\end{eqnarray*}
By~\eqref{constdistance} and $\delta_1\le c_1\sqrt{\kappa}$ we now have 
\begin{equation}
|K(z)-\Kz(z)|\leq  Cc_1\sqrt{\kappa+\eta}
\end{equation}
for all $\eta\in [\eta',1]$.

Choosing $c_1$ so small so that $Cc_1\le \frac{1}{2}c$ where $c$ is  the constant 
obtained in the estimate \eqref{distKz} in Lemma~\ref{Kzero},
we have
\begin{eqnarray}
\frac{1}{|1-K(z)|}&\leq & \frac{C}{\sqrt{\kappa+\eta}}.\label{boundfrac}
\end{eqnarray}
Using~\eqref{diffK} we  have
\begin{eqnarray}
|\munp(z)-\mud(z)|\leq  \frac{ C \delta_1}{\sqrt{\kappa+\eta}}\label{closestart}
\end{eqnarray}
for all $\eta\in[\eta',1]$ and with a sufficiently large constant $C$.
\end{proof}

\begin{lemma}\label{discrete_int}
 Let $\mu$ be a probability measure supported on some interval 
$[u,v]\subset \R$ and let $\{\tau_k\}_{k=1}^n$ be real numbers such that 
\begin{equation*}
 \int_{u}^{\tau_k}d\mu(\tau)=\frac{k}{n}
\end{equation*}
for $k=1,\ldots,n$ 
and assume $m\in \C^+$.
Then
\begin{equation*}
 \left|\on\sum_{k=1}^n \frac{1}{1+m\tau_k}-\int \frac{1}{1+m\tau}d\mu(\tau)\right|
\leq \frac{1}{n}\frac{|v-u|\cdot|m|}{\inf_{t\in [u,v]}|1+m t|^2}.
\end{equation*}
If $\{t_k\}_{k=1}^n$ is another set of points in $\R$, then
\begin{equation*}
 \left|\on\sum_{k=1}^n\frac{1}{1+mt_k}-\on\sum_{k=1}^n\frac{1}{1+m\tau_k}\right|
\leq \frac{1}{n}\frac{|m|}{\inf_{t\in[u,v]}|1+mt|^2 }\sum_{k=1}^n|t_k-\tau_k|.
\end{equation*}

\end{lemma}
\begin{proof}
For an arbitrary differentiable $f$ we have
\begin{eqnarray*}
\left|\on\sum_{k=1}^nf(\tau_k)-\int f(\tau)d\mu(\tau)\right|&
\le &\frac{|b-a|}{n}\sup_{t\in[u,v]}|f'(t)|.
\end{eqnarray*}
Here we have
\begin{eqnarray}
 \sup_{t\in[u,v]}\left|\frac{d}{dt}\; \frac{1}{1+mt}\right|&
\leq & \frac{|m|}{\inf_{t\in [u,v]}|1+tm|^2}.\label{real}
\end{eqnarray}
This proves the first claim. 
The second claim is proven similarly.  
\end{proof}

\subsection{A bound on $|\dj|$}\label{dj}

For any  $z\in \G$ we set
\begin{equation}
M(z):=\max\left( \|(\mnu(z)T+I)^{-1}\|,\max_{1\leq j\leq  N   }\|  (\mju(z)T+I)^{-1}\|\right).\label{defM}
\end{equation}

\begin{lemma}\label{bounddj}  Suppose that for some $z\in \G$ we have $M(z)\le C_0$ with some constant $C_0$.
Then there exists a  constant $C$  such that 
\begin{equation*}
\P\Big(|\dj(z)|> \frac{M(\log n)^{C}}{\sqrt{n\eta}} \big|\B,\To,\Tt,\X\Big)\leq Cn^{-\log\log n}
\end{equation*}
for $j=1,..., N  $ whenever $n\eta \ge M^2$.  
\end{lemma}

We start with a short lemma.
\begin{lemma}\label{diffm}
For $\mn$ and $\mj$ as defined in~\eqref{mnmj} and $z=E+i\eta$, 
\begin{equation*}
|\mn(z)-\mj(z)|\leq \frac{\pi}{\eta n}
\end{equation*}
for all $n\in \mathbb{N}$ and all $E\in\R$. 
\end{lemma}
\begin{proof}
Using $F^n$ and $F^{(j)}$ to denote the corresponding distribution functions, 
by Theorem~A.44 of~\cite{BS10},
\begin{equation*}
\sup_{t\in\R}|F^n(t)-F^{(j)}(t)|\leq \on.
\end{equation*}
Then
\begin{equation*}
|\mn(z)-\mj(z)|\leq\on\int\frac{dx}{|x-z|^2}=\frac{\pi}{\eta n}.
\end{equation*}

\end{proof}

\begin{proof}{\bf of Lemma~\ref{bounddj}}
We need to bound the terms~\eqref{termA}-\eqref{termD} and will use 
Lemmas~\ref{boundsB} and~\ref{Kzero} repeatedly. 
For~\eqref{termA} we use Lemma~\ref{diffm} and obtain
\begin{eqnarray*}
\on |\E \tr \Bjzi(\mnu(z)T+I)^{-1}(\mnu(z)-\mju(z))T(\mnu(z)T+I)^{-1} |
&\leq & \frac{CM^2}{n\eta}\log^2 n.
\end{eqnarray*}
Similarly, 
\begin{eqnarray*}
&&\hspace{-1in}\on \|\Bjzi(\mnu(z)T+I)^{-1}(\mnu(z)-\mju(z))T(\mnu(z)T+I)^{-1} \|_{HS}\\
&\leq & \frac{CM^2}{n\eta}\on(\tr |\Bj-zI|^{-2})^{\oh}\\
&\leq&\frac{CM^2}{(n\eta)^{3/2}},
\end{eqnarray*}
where the last inequality follows from~\eqref{boundHS}. 
Using Lemma B.2 of~\cite{EYY11}, 
\begin{equation*}
\P\left(|\eqref{termA}|\geq  \frac{M(\log n)^{C}}{\sqrt{n\eta}} \right)\leq n^{-\log\log n}. \label{secondprob}
\end{equation*}
For~\eqref{termB} we need to bound the absolute value of 
\begin{equation}
\frac{b}{N}\sum_{k,l=1}^{N} \Big[T^{\oh}(\Bj-zI)^{-1}(\mju(z)T+I)^{-1}T^{\oh}
 \Big]_{k,l}(b_k\bo_l-\E b_k\bo_l),\label{randomquadform}
\end{equation}
where $b_k$ denotes the $k^{th}$ entry of the vector $x_j$.   On the set $\B$, 
the second inequality in~\eqref{eqB}  gives the following bound on the Hilbert-Schmidt norm
\begin{eqnarray*}
\on\| \Th(\Bj-zI)^{-1}(\mju(z)T+I)^{-1}\Th\|_{HS}
&\leq& \frac{C}{n} \|(\mju(z)T+I)^{-1}\|\| (\Bj-zI)^{-1}\|_{HS}\\
&\leq &\frac{CM}{n}  \|(\Bj-zI)^{-1}\|_{HS}\\
&\leq& \frac{CM}{\sqrt{n\eta}} .\label{refHS}
\end{eqnarray*}
Thus, 
by Lemma~B.2 of~\cite{EYY11}, 
\begin{equation}
\P\Big(|\eqref{termB}|>  M \frac{(\log n)^{C}}{\sqrt{n\eta}}\Big)\leq Cn^{-\log\log n}.\label{boundprobquad}
\end{equation}

For~\eqref{termC}, using Lemma~\ref{diffm},~\eqref{boundtrace} and the resolvent identity,  we have 
\begin{eqnarray*}
|\eqref{termC}|
&\leq & \frac{C}{n}\| (\mnu(z)T+I)^{-1}-(\mju(z)T+I)^{-1} \|\; \tr |(\Bj-zI)^{-1}|\\
&\leq & \frac{CM^2}{ n}   \frac{1}{n\eta}\tr |(\Bj-zI)^{-1}|\\
&\leq &  \frac{CM^2\log^2 n}{n\eta }.
\end{eqnarray*}

For~\eqref{termD}, using the resolvent identity we see that 
\begin{equation}
\big((\Bj-zI)^{-1} -(\Bn-zI)^{-1}\big)(\mnu(z)T +I)^{-1}T =\Bjzi \rj\rjs\Bnzi(\mnu(z)T+I)^{-1}T \label{fortermD}
\end{equation}
has rank one, so that a bound on the norm of~\eqref{fortermD} gives a bound on its trace. 
We then return to the expression on the left of~\eqref{fortermD} and use 
$\|\Bjzi\|,\|\Bnzi\|\leq \eta^{-1}$ to obtain
\begin{equation*}
|\eqref{termD}|\leq \on\|((\Bj-zI)^{-1} -(\Bn-zI)^{-1})(\mnu(z)T +I)^{-1}T  \|
\leq  \frac{CM}{n\eta}.
\end{equation*}
Combining these four bounds just obtained 
gives  the statement of the lemma.  
\end{proof}

\subsection{Proof of Theorem~\ref{thmprod}}\label{proofprod}

We isolate   three parts of the proof of Theorem~\ref{thmprod} in the following lemmas.

\begin{lemma}\label{bounddenom}
Let $z=E+i\eta$ with $E\in (\mlow,\mhigh) $ and
 $\eta>n^{-1}\kappa^{-2}(\log n)^{4C_T\log\log n}$, and assume that $\To$ and $\X$ hold. 
 Then with probability at least $1-n^{-c\log\log n}$, 
\begin{equation}
\min_{1\leq j\leq  N  }|1+\rjs\Bjzi\rj|\geq \frac{1}{4\mhigh|\mn(z)|}.\label{eqbounddenom}
\end{equation}
\end{lemma}
The proof of Lemma~\ref{bounddenom} is given following the proof of 
Theorem~\ref{thmprod}. 

\begin{lemma}\label{start}
Assume $E\in (\mlow,\mhigh)$ and $ \left(\lp-\frac{1}{E+1}\right)\left(\frac{1}{E+1}-\lm\right)\geq \kappa  $ 
and that $\To$ and $\X$ hold. 
Then for any  sufficiently small constant $c_1>0$ 
there exists  $c>0$ depending only on $c_1$ 
such that with probability at least $1-n^{-c\log\log n}$
\begin{equation}\label{gg1}
|\mP(E+i)-\mn(E+i)| \le  \frac{ (\log n)^{2C_T \log\log n}}{ \sqrt{n}}.
\end{equation}
In particular, for  any  sufficiently small constant $c_1>0$ 
\begin{equation}\label{gg2}
|\mP(E+i)-\mn(E+i)| \le c_1\sqrt{\kappa+1}.
\end{equation}
Moreover, we also have
\begin{equation}
 \min_{1\leq j\leq  N   }   |1+\rjs(\Bj-(E+i))^{-1} \rj | \geq   c_1 
\label{inequalities}
\end{equation}
holds for all $n$ large enough. 
\end{lemma}

\begin{proof}
Let $\{\mu_k\}_{k=1}^{n-1}$ and $\{v_k\}_{k=1}^{n-1}$ denote the eigenvalues and corresponding normalized eigenvectors of $\Bj$. 
In the following we use that the eigenvalues of $\Bj$ and $\Bn$ interlace (see page 82 of \cite{Bha97}). 
Given  that we are on the set $\To\cap \X$ (see ~\eqref{eqX} and~\eqref{eqTone}), we have  
\begin{eqnarray}
\Im \rjs (\Bj -(E+i)I)^{-1}\rj &\geq& \sum_{k=1}^{n-1} \frac{|v_k^* \rj|^2}{(8b\aM+\mhigh)^2+1}
\geq  c\sum_{k=1}^{n-1} |v_k^* \rj|^2 \nonumber \\
&=&c\|\rj\|^2
=c\|Tx\|^2
\geq c\|x\|^2\geq c. \label{initialdenom}
\end{eqnarray}

A similar calculation together with~\eqref{under} implies 
$\Im \mnu(E+i),\Im\mju(E+i)\geq c$, hence,  recalling the definition \eqref{defM}, 
\begin{equation}
M=M(E+i)=\max\left(\|(\mnu(E+i)T+I)^{-1}\|,\max_{1\leq j\leq  N  }\|(\mju(z)T+I)^{-1}\|\right)\leq  C. 
\label{defM1}
\end{equation} 
Recalling the constant $C_T$ from~\eqref{eqTtwo}, we set 
\begin{equation}
s_\eta:=\frac{(\log n)^{\frac{3}{2} C_T\log\log n}}{\sqrt{n\eta}}. 
\label{defs}
\end{equation}
Using $M$ as just  estimated,   
for $n$ large enough the extra factors of $\log n$ in $s_1$ 
will exceed the constant necessary to satisfy the condition of 
Lemma~\ref{bounddj}.  
Thus, by Lemma~\ref{bounddj}  and \eqref{fourset}, 
$$   
|d_j(E+i)|<s_1, \qquad  j=1,\ldots,N,
$$
 with probability at least $1-Cn^{-\log\log n}$. 
Returning to~\eqref{eqdimn} and using the lower bound on the denominators given by~\eqref{initialdenom}, 
\begin{equation}
\left| \frac{a}{(E+i)N}\sum_{j=1}^N \frac{1}{1+ \rjs (\Bj-(E+i)I)^{-1}\rj }\dj(E+i)\right|
\leq C s_1.\label{boundpert}
\end{equation}
Using equation~\eqref{Silv1.2} and recalling that $\{t_1,\ldots,t_n\}$  are the eigenvalues of $T$, 
\begin{eqnarray}
\lefteqn{\left|\mn(E+i)-\int \frac{1}{\lambda(1-\oa -\oa (E+i)\mn(E+i))-(E+i)} \fdl \right|   }\label{firsttobound}\\
&=&\left|\mn(E+i)-\frac{1}{E+i}\int \frac{1}{1+\lambda \mnu(E+i)}\fdl  \right|\nonumber\\
&\leq&\left|\mn(E+i)-\frac{1}{E+i}\on\sum_{k=1}^n \frac{1}{1+t_k \mnu(E+i)} \right|\label{termaa}\\
&&+ \left|\frac{1}{E+i}\on\sum_{k=1}^n \frac{1}{1+t_k \mnu(E+i)}  -\frac{1}{E+i}\int \frac{1}{1+\lambda \mnu(E+i)}\fdl  \right|.\label{termbb}
\end{eqnarray}
For~\eqref{termaa} we use~\eqref{eqdimn} and~\eqref{boundpert} to obtain   
\begin{equation*}
\left|\mn(E+i)-\frac{1}{E+i}\on\sum_{k=1}^n \frac{1}{1+t_k \mnu(E+i)} \right|\leq C s_1.
\end{equation*}
For~\eqref{termbb} we use~\eqref{eqTtwo},  both parts of Lemma~\ref{discrete_int}, 
and the bounds~\eqref{range},  \eqref{defM1}  to obtain 
\begin{eqnarray*}
\lefteqn{\left|\frac{1}{E+i}\on\sum_{k=1}^n \frac{1}{1+t_k \mnu(E+i)}
  -\frac{1}{E+i}\int \frac{1}{1+\lambda \mnu(E+i)}\fdl  \right|}\hspace{2in}\\
&\leq &\frac{C}{n} +\frac{(\log n)^{C_T\log\log n}}{n}\\
&\leq &\frac{(\log n)^{C_T\log\log n}}{n}.
\end{eqnarray*}
Since  $s_1>n^{-1}(\log n)^{C_T \log\log n}$  and $s_1\le c_1\sqrt{\kappa}$ if $n$ is sufficiently large, we have 
\begin{equation}
\eqref{firsttobound}\leq C s_1, \label{Cszero}
\end{equation}
so that by inequality~\eqref{claima} of Lemma~\ref{stability} we have 
\begin{equation}
|\mP(E+i)-\mn(E+i)|\leq C s_1\le  \frac{ (\log n)^{2C_T \log\log n}}{ \sqrt{n}}.\label{ineqone}
\end{equation}
This yields \eqref{gg1} and 
for $n$ large enough we also get \eqref{gg2}.
From~\eqref{ineqone}, \eqref{range}  and~\eqref{eqbounddenom} we have the bound 
\begin{equation*}
\min_{1\leq j\leq  N  }|1+\rjs(\Bj-(E+i)I)^{-1}\rj|\ge \frac{1}{4C\mu_+}
\end{equation*} 
with probability at least $1-n^{-c\log\log n}$. Choosing $c_1$ small enough, this 
 yields~\eqref{inequalities}.  
\end{proof}

\begin{lemma}\label{continuity}  
Assume $E\in (\mlow,\mhigh)$ and $ \left(\lp-\frac{1}{E+1}\right)\left(\frac{1}{E+1}-\lm\right)\geq \kappa  $
 and choose an $\eta' \in[n^{-1}\kappa^{-2}(\log n)^{4C_T\log\log n},1]$, 
where $C_T$ is the constant appearing in~\eqref{eqTtwo}. Set $\eta'':=\eta' - n^{-2}$. 
Assume  that $\To $ and $\X$ hold and that  with a sufficiently small constant $c_1$ 
\begin{equation}
|\mP(E+i\eta)-\mn(E+i\eta )| \leq   c_1 \sqrt{\kappa +\eta}
\label{threeconditions}
\end{equation}
holds
for all $\eta\in[\eta',1]$ with some probability at least $1-P(n)$.
Then with a probability at least $1-P(n)-n^{-c\log\log n}$, we have
\begin{equation}\label{ggo}
|\mP(E+i\eta)-\mn(E+i\eta)|  \le \frac{ (\log n)^{2C_T \log\log n}}{ \sqrt{n\eta\kappa}}
\end{equation}
 for all $\eta\in [\eta'',1]$. In particular, provided $\eta''\geq n^{-1}\kappa^{-2} (\log n)^{4 C_T \log\log n} $, 
 the bound~\eqref{threeconditions} 
holds for all $\eta\in [\eta'',1]$, with a probability  at least $1-P(n)-n^{-c\log\log n}$.
\end{lemma}

\begin{proof}
By a trivial continuity argument, first  we prove that 
\begin{equation}
|\mP(E+i\eta)-\mn(E+i\eta )| \leq  2c_1\sqrt{\kappa+\eta}
\label{threeconditionssecond}
\end{equation}
holds for all $\eta\in [\eta'',1]$ with probability at least $1-P(n)$. 
Indeed, the functions  $\mP(E+i\eta)$ and $\mn(E+i\eta)$ are  Lipschitz  continuous for fixed $E$ and $\eta\in [\eta'',1]$  
with derivatives bounded by $(1/\eta'')^2$. 
In particular, the derivative of $\mn$ with respect to $\eta$ 
is uniformly bounded on $[n^{-1}\kappa^{-2}(\log n)^{4C_T\log\log n},1]$ by $Cn^{2}\kappa^4(\log n)^{-8C_T\log\log n}$. 
By this continuity and the proximity of $\mn(E+i\eta')$ and $\mP(E+i\eta')$ 
given by~\eqref{threeconditions}, 
the inequality~\eqref{threeconditionssecond}  holds  
for all $\eta\in[\eta'',1]$.

Now we show that  the  stronger  estimate~\eqref{threeconditions}
 can be regained from the weaker estimate~\eqref{threeconditionssecond}   for all $\eta\in [\eta'',1]$.  
Assuming that $c_1$ is sufficiently small,  Lemma~\ref{diffm}, the first inequality of~\eqref{threeconditionssecond} 
and  inequality~\eqref{mtplusI} of Lemma~\ref{Kzero} 
along with the spectral theorem give the 
bound $M(E+i\eta)<C$ for all $\eta\in[\eta'',1]$.
Recalling the definition of $s_\eta $ from~\eqref{defs},  
we use the bound on $M$  and Lemma~\ref{bounddj} for large enough $n$ to obtain
\begin{equation*}
\max_{1\leq j\leq  N   }|d_j(E+i\eta)|<  s_\eta 
\end{equation*}
with probability at least $1-P(n)-n^{-c\log\log n}$ for all $\eta\in[\eta'',1]$. 

Using \eqref{range} and \eqref{threeconditionssecond}, we get
that $|\mn(E+i\eta )|\ge \frac{1}{2}c$ 
 for  all $\eta\in[\eta'',1]$ if $c_1$ is sufficiently small,
where $c$ is the constant from Lemma~\ref{Kzero}.
 Then 
from Lemma~\ref{bounddenom}, we see that, with probability at least $1-P(n)-n^{-c\log\log n}$,
\begin{equation}\label{lowbjb}
 \min_{1\leq j\leq  N  }  |1+\rjs(\Bj-(E+i\eta))^{-1} \rj | \geq c'
\end{equation}
with some positive constant $c'$. 
Using~\eqref{lowbjb}
 to bound the denominators 
and an argument just like the one used for the point 
$E+i$ in Lemma~\ref{start}, 
\begin{equation}\label{L39}
\left|\mn(E+i\eta)-\int \frac{1}{\lambda(1-\oa -\oa (E+i\eta)\mn(E+i\eta))-(E+i\eta)} \fdl \right|  \leq  Cs_\eta 
\end{equation}
for all $\eta\in[\eta'',1]$. 
By Lemma~\ref{stability} we obtain 
\begin{equation}\label{goodbound}
|\mP(E+i\eta)-\mn(E+i\eta)|\leq \frac{C s_\eta}{\sqrt{\kappa+\eta}}
  \le \frac{ (\log n)^{2C_T \log\log n}}{ \sqrt{n\eta\kappa}}
\end{equation}
for all $\eta\in [\eta'',1]$  with probability at least $1-P(n)-n^{-c\log\log n}$.
Since 
$$
\eta\ge \eta''>n^{-1}\kappa^{-2}(\log n)^{4C_T \log\log n},
$$
 we have
$$
s_\eta < (\kappa+\eta)(\log n)^{-2C_T \log\log n},
$$ 
which yields~\eqref{threeconditions} with a very high probability for all 
$\eta\in [\eta'',1]$ and for sufficiently large $n$.
\end{proof}

\begin{proof}{\bf of Theorem~\ref{thmprod}}
The proof of $i)$ follows directly from the previous two lemmas.
Set $\eta_p = 1 - pn^{-2}$, $p=0,1,2,\ldots $, then~\eqref{thmclaimone} is
proved in Lemma~\ref{start} for $\eta =1$ with probability at least $1-n^{-c\log\log n}$. 
Lemma~\ref{continuity} then implies that
\begin{equation}\label{gg4}
|\mP(E+i\eta)-\mn(E+i\eta)|
  \le \frac{ (\log n)^{2C_T \log\log n}}{ \sqrt{n\eta\kappa}}
\end{equation}
holds for all $\eta\in [\eta_1, \eta_0=1]$ with the possible
exception of a set of probability $2n^{-c\log\log n}$.
Iterating Lemma~\ref{continuity}, we get that
\eqref{gg4} holds for all $\eta \in [\eta_p ,1]$
with probability at least $1- (p+1)n^{-c\log\log n}$
as long as $\eta_p\ge n^{-1}\kappa^{-2} (\log n)^{4 C_T \log\log n} $.
This proves  the first part of Theorem~\ref{thmprod} for a fixed energy $E$.
To take care of all energies simultaneously,  we use
that the derivative of $\mn$ is uniformly bounded by $C\eta^{-2}$.
Thus  we can discretize the energy range to $Cn^2$ points and 
take the union bound to obtain $i)$ of Theorem~\ref{thmprod}.

Now we prove $ii)$. 
Set 
\begin{equation*}
\rho_\eta(E)=\frac{1}{\pi}\Im \mn(E+i\eta)=\frac{1}{\pi n}\sum_{k=1}^n\frac{\eta}{(\mu_k-E)^2+\eta^2} .
\end{equation*}
Given Lemma~\ref{bins}, the argument for Corollary~2.2 in~\cite{ESY09} gives 
\begin{equation*}
\P \left( \sup_{E\in\Eck}\rho_{\eta}(E)>K  \right)\leq 1-Cnn^{-\log\log n} 
\end{equation*}
where $K$ is the constant in Lemma~\ref{bins}. 
Using this inequality,  $ii)$ follows from the argument given to prove the analogous claim 
in Corollary~4.2 of~\cite{ESY09}. 
\end{proof}

\begin{proof}{\bf of Lemma~\ref{bounddenom}}
From~\eqref{keyident2} we have  
\begin{eqnarray}
\mn(z)=\frac{-1}{1+\on \tr T\Bnzi  }\frac{1}{z}(1-A-B)\label{tracedenom}
\end{eqnarray}
where 
\begin{equation*}
A:= \frac{1}{N}\sum_{j=1}^N\frac{\on \tr T \Bjzi -\on \tr T\Bnzi}{1+\rj^*\Bjzi\rj}
\end{equation*}
and
\begin{equation*}
B:= \frac{1}{N}\sum_{j=1}^N\frac{\rj^*\Bjzi \rj-\on\tr T\Bjzi}{1+\rj^*\Bjzi\rj}.
\end{equation*}
We first give a bound on $A$. 
Using~\eqref{ident1} and the resolvent identity we have 
\begin{equation*}
\tr T \Bnzi - \tr T\Bjzi=\frac{1}{1+\rj^*\Bjzi\rj}\rjs \Bjzi T\Bjzi \rj.
\end{equation*}
Using Lemma~\ref{boundsB}, 
\begin{eqnarray}
\left|\E\on \rjs \Bjzi T\Bjzi \rj\right|\leq \frac{C}{n^2}\tr |\Bj-zI|^{-2}\leq \frac{C}{n\eta},
\end{eqnarray}
which for large $n$ can be made smaller than any constant. 
Similarly, 
\begin{eqnarray*}
\frac{1}{n^2}\sqrt{\sum_{k,l}|(T|\Bj-zI|^{-2})_{k,l}|^2}&\leq&\frac{C}{n^2}\sqrt{\tr |\Bj-zI|^{-4}}\\
&\leq &\frac{C}{n^2}\sqrt{\frac{1}{\eta^2}\tr |\Bj-zI|^{-2}}\\
&\leq&\frac{C}{(n\eta)^{3/2}},
\end{eqnarray*}
which is also smaller than any constant for large $n$. 
Using the resolvent identity, Lemma B.2 of~\cite{EYY11} and a union bound, for any constant $\epsilon>0$, 
\begin{small}\begin{eqnarray}
\lefteqn{\P\left(\max_{1\leq j\leq  N  }\left| \on \tr T \Bnzi -\on \tr T\Bjzi\right|\geq \epsilon \right)} \label{boundA}
\\
&\leq &\P\left(\max_{1\leq j\leq  N  }\left| \on
 \tr T \Bnzi\rj\rjs\Bjzi- \E\on \rjs \Bjzi T\Bjzi \rj \right|\geq \frac{\epsilon}{2}  \right)\nonumber\\
&\leq  & n^{-c\log\log n}.\nonumber
\end{eqnarray}\end{small}
Now we look at $B$. 
By Lemma~\ref{boundsB},  
$\|T(\Bj-zI)^{-1}\|_{HS}\leq C \sqrt{\frac{n}{\eta}}$. 
In this case,  again for any constant $\epsilon>0$ and for $n$ large enough, 
\begin{equation*}
\frac{(\log n)^{C}}{n}\|T\Bjzi\|_{HS}\leq C\frac{(\log n)^{C}}{\sqrt{n\eta}}\leq \epsilon.
\end{equation*}
Therefore, by Lemma~B.2 of~\cite{EYY11}
\begin{equation}
\P\left(\max_{1\leq j\leq n}| \rjs \Bjzi\rj -\on\tr T\Bjzi|> \epsilon \right)\leq n^{-c\log\log n}.\label{boundB}
\end{equation} 
If $|1+\rjs\Bjzi\rj|>c$, then by setting $\epsilon=c^2$ in~\eqref{boundA} and~\eqref{boundB}, 
with probability at least $1- n^{-c\log\log n}$, we have
$|z|^{-1}(1-|A|-|B|)\geq (2\mhigh)^{-1}$, 
so that, by~\eqref{tracedenom}, 
\begin{equation*}
\left|\frac{1}{1+\on\tr T\Bnzi}\right|\leq 2\mhigh |\mn(z)|.
\end{equation*}
By~\eqref{boundA} and~\eqref{boundB}, with probability at least $1-n^{-c\log\log n}$, 
\begin{equation*}
\max_{1\leq j\leq  N }\left| \frac{1}{1+\rjs \Bjzi\rj}\right|\leq 4\mhigh|\mn(z)|.
\end{equation*}
\end{proof}

\begin{proof}{\bf of Theorem~\ref{deloc}}
Recall that  $\Xo$ denotes the $n\times  (N-1)  $ matrix obtained by removing the first column of $X$. 
Consider the component $v_1$ of $v$.  
Following the proof of Theorem 1.2 in~\cite{ESY09a}, 
\begin{eqnarray}
|v_1|^2&=&(1+x_1^* \Tih \Xo(\mu-\Xo^*\Ti \Xo)^{-2}\Xo^* \Tih x_1)^{-1}.\label{v}
\end{eqnarray}
Let $\Tih\Xo$ have the singular value decomposition
\begin{equation*}
\Tih\Xo=\sum_{k=1}^{n-1}\sqrt{\mu^{(1)}_k}u_kw_k^*.
\end{equation*}
When $\To$ and $\X$ occur, 
$\min(\mu^{(1)}_1,\ldots,\mu^{(1)}_{n-1})\geq \oh\am(1-\frac{1}{\sqrt{b}})^2$. 
As in~\cite{ESY09a}, we now partition $\Eck$ into subintervals of length $\eta$. 
If $\mu$ is in the interval $I_a=[a-\frac{\eta}{2},a+\frac{\eta}{2}]$, then 
\begin{eqnarray*}
\eqref{v}&=&\left(1+\sum_{k=1}^{n-1}\frac{\mu^{(1)}_k|u_k^* x_1|^2}{(\mu-\lambda_k)^2}\right)^{-1}\\
&\leq &\left( \sum_{k:\lambda_k\in I_a}\frac{\mu^{(1)}_i|u_k^* x_1|^2}{(\mu-\lambda_k)^2}\right)^{-1}\\
&\leq&\left(\oh\am \left(1-\frac{1}{\sqrt{b}}\right)^2\right)^{-1} \left( \frac{1}{\eta^2}\sum_{k:\lambda_k\in I_a}|u_k^* x_1|^2\right)^{-1}
\end{eqnarray*}
so that, setting $\xi_k=|\sqrt{n} x_1^*u_1|$,   
\begin{equation}
\P(|v_1|>t)\leq \P\left(\sum_{k:\lambda_k\in I_a}\xi_k\leq  \left(\oh\am (1-b^{-\oh})^2\right)\frac{ \eta^2n  }{t^2   }\right).\label{probv1}
\end{equation}
The eigenvalues of $\Tih\Xo\Xo^*\Tih$ and $\Tih X X^*\Tih$ interlace (see page~82 of~\cite{Bha97}), so that 
the number of eigenvalues of each matrix in an interval differs at most by two. 
We use Lemma~\ref{like4.7ofWegner} and an argument similar to the proof of Lemma~\ref{bins} to 
obtain the optimal  $t$, $t=\eta^{\oh}$,  and a bound for~\eqref{probv1}. 
We partition $\Eck$ into at most $n$ subintervals of the form $I_a$ and take the 
union bound for the terms of the form~\eqref{probv1} corresponding to the subintervals and the $b n$ indices. 
\end{proof}
\section{General MANOVA Matrices }\label{genmanova}

Now we turn to the proof of Theorem~\ref{thmmanova}, which follows easily from Theorem~\ref{thmprod}.

\begin{proof}{\bf of Theorem~\ref{thmmanova}}
We recall the observations that 
\begin{equation}
\Matrix \label{M1}
\end{equation}
and 
\begin{equation}
(Y \Ys)^{\oh}(X\Xs+Y\Ys)^{-1}(Y\Ys)^{\oh}\label{M2}
\end{equation}
have the same eigenvalues and 
\begin{equation*}
(Y \Ys)^{\oh}(X\Xs+Y\Ys)^{-1}(Y\Ys)^{\oh}=(I+(Y\Ys)^{-\oh}X\Xs(Y\Ys)^{\oh})^{-1}.
\end{equation*}
Note that by factoring $YY^*$ as two factors of $(YY^*)^{\oh}$ rather than $Y$ and $Y^*$, 
we avoid changing the number of zero eigenvalues. 
If $(Y\Ys)^{-\oh}X\Xs(Y\Ys)^{-\oh}$ has  eigenvalues $\mu_1, \mu_2, \ldots \mu_n$, then~\eqref{M2}, 
and hence~\eqref{M1}, has  eigenvalues $\lambda_k=(1+\mu_k)^{-1}$, $k=1,2,\ldots, n$,
i.e. the correspondence $\lambda = (1+\mu)^{-1}$ maps one set of eigenvalues into the other one.
 Since the eigenvalues of $(Y\Ys)^{-\oh}X\Xs(Y\Ys)^{-\oh}$ are positive, this correspondence is regular. 
The correspondence also maps the interval $[\mlow,\mhigh]$ to $[\lm,\lp]$ and the regions $\Eek$ to $\EMe$ 
and $\Eck$ to $\EMc$. 
The same transformation gives the correspondence between $f_p$ and $f_M$:
$$
   f_p(\mu) = \frac{1}{(1+\mu)^2} f_M\Big( \frac{1}{1+\mu}\Big).
$$
An easy calculation shows the relation between the Stieltjes transforms. Setting $z'= z^{-1}-1$,
after a change of variables, we have
$$
   m_p(z') = \int \frac{1}{\mu-z'}f_p(\mu) d\mu  =  - \int \frac{\lambda z}{\lambda- z}
f_M(\lambda) d\lambda = -z -z^2 m_M(z).
$$
Similarly
$$
   m_{n,p}(z') = \frac{1}{n}\sum_{k=1}^n \frac{1}{\mu_k-z'} 
  = - \frac{z}{n}\sum_k \frac{\lambda_k }{\lambda_k-z} = - z- z^2 m_{n,M}(z).
$$
so
\begin{equation}
  m_{n,M}(z) - m_M(z) = -z^{-2}( m_{n,p}(z') -  m_p(z')).
 \label{MMM}
\end{equation}
(Strictly speaking, we  defined Stieltjes transforms $m(z)$
for $z\in \C^+$, but the formula \eqref{ST} clearly defines
it for all $z\in \C\setminus\R$ and we have $\overline{m(z)} = m(\bar z)$.) 
Since the sets $\Eck$ and $\EMc$ are separated away from zero,
we always have a positive lower bound on $z$ and $z'$. 
Notice that $\Im z' = - |z|^{-2}\Im z$, therefore $\Im z'$ and $\Im z$ are
comparable.  Thus part $i)$ of  Theorem~\ref{thmmanova} follows from 
  \eqref{MMM} and part $i)$ of Theorem~\ref{thmprod}.
One can similarly conclude part $ii)$ of  Theorem~\ref{thmmanova}
from part $ii)$  of Theorem~\ref{thmprod}.
 This proves Theorem~\ref{thmmanova}.
\end{proof}

\appendix
\section{Computing $m_M$}\label{mNcomp}
We refer to~\cite{Far11} for the calculation of $m_M(z)$. 
There the Stieltjes transform of 
\begin{equation*}
f(x)=\frac{a}{a+b}\delta(x)+ \frac{\sqrt{(x-\lm)(\lp-x)}}{2\pi x(1-x)}I_{[\lm,\lp]}+\Big(1-\frac{b-1}{b}\Big)\delta(x-1)
\end{equation*}
is obtained, which we will denote $m_f(z)$ here. 
The result is 
\begin{equation}\label{mf}
   m_f(z) = \frac{z-\frac{a+1}{a+b}+\sqrt{z^2-(2\frac{a+1}{a+b}-\frac{a+1}{(a+b)^2})z+(\frac{a-1}{a+b})^2}}{2z(1-z)}.
\end{equation}
This is seen by setting $p=\frac{a}{a +b}$ and $q=\frac{1}{a+b}$ 
in the notation of~\cite{Far11} (note, though, that the point mass at zero  in $f(x)$ has been removed in~\cite{Far11}). 
We then have that 
\begin{equation*}
m_M(z)=(a+b)\left(m_f(z)+\frac{a}{a+b}\frac{1}{z}-\frac{b-1}{a+b}\frac{1}{1-z}\right).
\end{equation*}
Inserting \eqref{mf} into this formula, we obtain \eqref{mN}.

\def\cprime{$'$} \def\cprime{$'$} \def\cprime{$'$} \def\cprime{$'$}


\begin{thebibliography}{10}

\bibitem{BS10}
Z.~D. Bai and J.~W. Silverstein.
\newblock {\em Spectral analysis of large dimensional random matrices}.
\newblock Springer, New York, 2010.

\bibitem{Bha97}
R.~Bhatia.
\newblock {\em Matrix analysis}.
\newblock Springer-Verlag, New York, 1997.

\bibitem{ESY09a}
L.~Erd{\H{o}}s, B.~Schlein, and H.-T. Yau.
\newblock Local semicircle law and complete delocalization for {W}igner random
  matrices.
\newblock {\em Comm. Math. Phys.}, 287(2):641--655, 2009.

\bibitem{ESY09}
L.~Erd{\H{o}}s, B.~Schlein, and H.-T. Yau.
\newblock Semicircle law on short scales and delocalization of eigenvectors for
  {W}igner random matrices.
\newblock {\em Ann. Probab.}, 37(3):815--852, 2009.

\bibitem{ESY10}
L.~Erd{\H{o}}s, B.~Schlein, and H.-T. Yau.
\newblock Wegner estimate and level repulsion for {W}igner random matrices.
\newblock {\em Int. Math. Res. Not. IMRN}, (3):436--479, 2010.

\bibitem{ESY11}
L.~Erd{\H{o}}s, B.~Schlein, and H.-T. Yau.
\newblock Universality of random matrices and local relaxation flow.
\newblock {\em Invent. Math.}, 185(1):75--119, 2011.

\bibitem{ESYY10}
L.~Erd{\H{o}}s, B.~Schlein, H.-T. Yau, and J.~Yin.
\newblock The local relaxation flow approach to universality of the local
  statistics for random matrices.
\newblock {\em Ann. Inst. H. Poincar\'e (B), Probab. Statist}.,
\newblock 48(1): 1-46, 2012.

\bibitem{EYY11}
L.~Erd{\H{o}}s, H.-T. Yau, and J.~Yin.
\newblock Bulk universality for generalized {W}igner matrices.
\newblock {\em Prob. Theor. Rel. Fields},
\newblock 154(1-2): 341-407, 2012.

\bibitem{EYY11a}
L.~Erd{\H{o}}s, H.-T. Yau, and J.~Yin.
\newblock Rigidity of eigenvalues of generalized {W}igner matrices.
\newblock {\em Adv. Math.},
\newblock 229(3): 1435-1515, 2012. 

\bibitem{Far11}
B.~Farrell.
\newblock Limiting empirical singular value distribution of restrictions of
  discrete {F}ourier matrices.
\newblock {\em Journal of Fourier Analysis and Applications}, 17:733--753,
  2011.

\bibitem{For10}
P.~J. Forrester.
\newblock {\em Log-gases and random matrices}.
\newblock Princeton University Press, Princeton, NJ, 2010.

\bibitem{MP67}
V.~A. Marchenko and L.~A. Pastur.
\newblock Distribution of eigenvalues in certain sets of random matrices.
\newblock {\em Mat. Sb. (N.S.)}, 72 (114):507--536, 1967.

\bibitem{Mui82}
R.~J. Muirhead.
\newblock {\em Aspects of multivariate statistical theory}.
\newblock John Wiley \& Sons Inc., New York, 1982.

\bibitem{Pec12}
S.~P\'ech\'e.
\newblock Universality in the bulk of the spectrum for complex sample
  covariance matrices.
\newblock {\em Ann. Inst. H. Poincar\'e Probab. Statist.}, 48 (1):80--106,
  2012.


\bibitem{PY13}
N.~S. Pillai and J.~Yin.
\newblock Universality of covariance matrices.
\newblock {\em Ann. Appl. Probab.}.
\newblock to appear.


\bibitem{Sil95}
J.~W. Silverstein.
\newblock Strong convergence of the empirical distribution of eigenvalues of
  large-dimensional random matrices.
\newblock {\em J. Multivariate Anal.}, 55(2):331--339, 1995.

\bibitem{SB95}
J.~W. Silverstein and Z.~D. Bai.
\newblock On the empricial distribution of eigenvalues of a class of large
  dimensional random matrices.
\newblock {\em J. Multivariate Anal.}, 55(2):175--192, 1995.

\bibitem{TV11b}
T.~Tao and V.~Vu.
\newblock Random covariance matrices: universality of local statistics of
  eigenvalues.
\newblock  {\em Ann. Probab.}, 40(3):1285-1315, 2012.

\bibitem{TV11a}
T.~Tao and V.~Vu.
\newblock Random matrices: universality of local eigenvalue statistics.
\newblock {\em Acta Math.}, 206(1):127--204, 2011.

\bibitem{Wac80}
K.~W. Wachter.
\newblock {The limiting empirical measure of multiple discriminant ratios.}
\newblock {\em Ann. Stat.}, 8:937--957, 1980.


\bibitem{Wig55}
E.~Wigner.
\newblock Characteristic vectors of bordered matrices with infinite dimensions.
\newblock {\em Ann. of Math. (2)}, 62:548--564, 1955.

\end{thebibliography}
\end{document}